\documentclass[reqno,twoside,11pt]{amsart}

\usepackage{amsmath,amsfonts,calrsfs,fullpage,amssymb,color,verbatim,eucal,yfonts,mathrsfs,marginnote, mathtools}

\newtheorem{Theorem}{Theorem}[section]

\newtheorem{Proposition}[Theorem]{Proposition}
\newtheorem{Lemma}[Theorem]{Lemma}
\newtheorem{Corollary}[Theorem]{Corollary}
\newtheorem{Remark}[Theorem]{Remark}

\newtheorem{Hypothesis}[Theorem]{Hypothesis}
\makeatletter
\@addtoreset{equation}{section}

\makeatother
 
\def\R{\mathbb R}
\def\N{\mathbb N}

\def\C{\mathbb C}
\def\E{\mathbb E}

\def\eps{\varepsilon}
\def\ds{\displaystyle}

\newcommand{\esssup}{\operatorname{ess\,sup}}

\title{\bf Time regularity for generalized Mehler semigroups}\date{}

\author[A. Lunardi]{Alessandra Lunardi}
\address{
Dipartimento di Scienze Matematiche, Fisiche e Informatiche\\
Universit\`a di Parma\\
Parco Area delle Scienze, 53/A\\
43124 Parma, Italy}
\email{alessandra.lunardi@unipr.it}



\begin{document}

 \begin{abstract}  
We study continuity and H\"older continuity of $t\mapsto P_tf$, where $P_t$ is a generalized Mehler semigroup in  $C_b(X)$, the space of the continuous and bounded functions from a  Banach space $X$ to $\R$, and $f\in C_b(X)$. The generators $L$ of such semigroups are realizations of   a class  of differential and pseudo-differential operators, both in finite and in infinite dimension. Examples of operators $L$ to which this  theory is applicable include Ornstein-Uhlenbeck operators with fractional diffusion in finite dimension, and Ornstein-Uhlenbeck operators with associated strong-Feller semigroups,  in infinite dimension. 
 \end{abstract}

 \maketitle

\section{Introduction}
Let $X$ be a real  Banach space, and let $C_b(X)$ denote the space of the continuous and bounded functions from $X$ to $\R$, endowed with the sup norm. 
A generalized Mehler semigroup $P_t$  is a semigroup of bounded operators in  $C_b(X)$ that may be represented by 
 \begin{equation}
 \label{P_t}
 P_tf(x) = \int_{X} f(T_tx + y)\mu_t(dy), \quad t\geq 0, \; f\in C_b(X), \; x\in X, 
 \end{equation}
where $T_t$ is a strongly continuous semigroup of bounded operators  on  $X$, and $\{\mu_t:\; t\geq 0\}$ is a family of Borel probability measures in $X$ such that $\mu_0= \delta_0$ (the Dirac measure at $0\in X$), $t\mapsto \mu_t$ is weakly continuous in $[0, +\infty)$,  and  
 \begin{equation}
 \label{semigruppo}
 \mu_{t+s} = (\mu_t \circ T_s^{-1})\ast \mu_s, \quad t, s >0, 
  \end{equation}
that is in fact an algebraic necessary and sufficient for $P_t$ be a semigroup. 

They arise as transition semigroups of (weak or mild) solutions to stochastic differential equations such as  
\begin{equation}
\label{stoch}
dX(t)= AX(t)dt + dL_t, \quad  t>0; \; X(0)=x,
\end{equation}
where $A:D(A)\subset X\rightarrow X$ is the infinitesimal generator of $T_t$,  and $\{ L_t: \;t\geq0\}$  is a L\'evy process in $X$. 
Detailed discussions about several features of such  semigroups, and about connections with the stochastic differential equations \eqref{stoch}, are in \cite{BRS, ACM, DL, DLScS, FR, LR1, LR2, ZL, OR, ORW, SchSun} and in the papers quoted therein; see also the review paper \cite{Ap}. 

It is not hard to see that for every $f\in C_b(X)$ and $x\in X$ the function $(t,x)\mapsto P_tf(x)$ is continuous in $[0, +\infty)\times X$ (e.g., \cite[Lemma 2.1]{BRS}). But $P_t$ is not strongly continuous in general, as well as its restriction to the space $BUC(X)$   of the bounded and uniformly continuous functions from $X$ to $\R$. The {\it generator } $L$ of $P_t$ is defined through its resolvent, 
 \begin{equation}
 \label{L}
 R(\lambda, L)f (x) = \int_0^{\infty} e^{-\lambda t}P_tf(x)\,dt, \quad \lambda >0, \;f\in C_b(X), \; x\in X, 
  \end{equation}
however it is not the infinitesimal generator in the standard sense.

In this paper we study continuity and H\"older continuity of the function $[0, +\infty)\mapsto C_b(X)$, $t\mapsto P_tf$, for $f\in C_b(X)$. 

As a simple first result, we show that for every $f\in C_b(X)$, $t\mapsto P_tf$ is continuous if and only if  $\lim_{\lambda \to +\infty} \lambda R(\lambda, L)f =f$, if and only if
$f\in \overline{D(L)}$. So, $\overline{D(L)}$ is the subspace of strong continuity of $P_t$. This characterization is shared by (not necessarily strongly continuous) analytic semigroups in general Banach spaces, see e.g. \cite[Sect. 2.1]{L}. But in general, generalized Mehler semigroups are not analytic in $C_b(X)$ nor in $BUC(X)$. Even worse, they are not necessarily eventually norm continuous: see \cite{PvN} for Ornstein-Uhlenbeck semigroups. 

Then we prove that for every $\alpha\in (0, 1)$, 
$t\mapsto P_tf$ is  $\alpha$-H\"older continuous if and only if  $sup_{\lambda >0} \|\lambda^\alpha LR(\lambda, L)f\|_{\infty} <+\infty$, if and only if
$f$ is in  the real interpolation space $(C_b(X), D(L))_{\alpha, \infty}$. This characterization is shared by strongly continuous or analytic  semigroups of nonnegative type. 

In fact, both characterizations are proved for a larger class of semigroups, and precisely for all bounded semigroups $P_t$ such that $(t,x)\mapsto P_tf(x)$ belongs to $C([0 +\infty)\times X)$ for every $f\in C_b(X)$. 

Going back to generalized Mehler semigroups, if $P_t$ enjoys  suitable smoothing assumptions it is possible to provide more explicit descriptions of both $\overline{D(L)}$ and $(C_b(X), D(L))_{\alpha, \infty}$. 

Concerning $\overline{D(L)}$, if  $P_t$ maps $C_b(X)$ to $BUC(X)$ we have 
$$\overline{D(L)} = \{f\in BUC(X):\; \lim_{t\to 0}\|f(T_t\;\cdot )- f\|_{\infty} =0\}. $$

Concerning $(C_b(X), D(L))_{\alpha, \infty}$, we need better smoothing properties of $P_t$. Precisely, we assume that the following hypothesis holds. 

\begin{Hypothesis}
\label{Hyp:LR}
Each  $\mu_t$ is Radon, and Fomin differentiable along the range of $T_t$. There exist $C>0$, $\omega \in \R$   such that denoting by $\beta_{t,h}$ the Fomin derivative of $\mu_t$ along $T_th$, we have 
\begin{equation}
\label{ipotesi_intro}
\| \beta_{t,h}\|_{L^1(X, \mu_t)}\leq \frac{C e^{\omega t}}{t^\theta }\|h\|,  \quad t>0, \,h\in X. 
\end{equation}
\end{Hypothesis}
Then it is possible to show that $P_t$ is strong-Feller (namely, if $f$ is Borel measurable and bounded, then $P_tf$, still defined by  \eqref{P_t}, belongs to $C_b(X)$) and moreover it  maps $C_b(X)$ into the space $C^k_b(X)$ of the $k$ times Fr\'echet differentiable functions with bounded derivatives up to the order $k$,  for $k\in \N$, $k<1/\theta$, and there is $C_k>0$ such that 
\begin{equation}
\label{essenziali}\| P_tf\|_{C^k_b(X)} \leq \frac{C_k}{t^{k\theta}}  \|f\|_{\infty}, \quad f\in C_b(X),\;0<t\leq 1, \; k<1/\theta.
\end{equation}
Hypothesis  \ref{Hyp:LR} was introduced in \cite{LR}, where Schauder type theorems were proved. Estimates \eqref{essenziali} are used to prove that for every $\alpha\in (0, 1)$  we have
\begin{equation}
\label{inclusione}
(C_b(X), D(L))_{\alpha, \infty} \subset \left\{ \begin{array}{ll}
C^{\alpha /\theta}_b(X), &  \alpha /\theta\notin \N, 
\\
\\
Z^{\alpha /\theta}_b(X), &  \alpha /\theta \in \N, 
\end{array}\right. 
\end{equation}
with continuous embedding. Here, for $k\in \N$, $\sigma\in (0, 1)$, $C^{k+\sigma}_b(X)$ consists  of the functions  $f\in C^k_b(X)$ having $\sigma$-H\"older continuous $k$-th order derivative, and   $Z^{k}_b(X)$ is the Zygmund space  defined  in Sect. 2. Besides estimates  \eqref{essenziali}, the proof of  \eqref{inclusione} relies on results proved in \cite{LR} and  on the continuous embeddings 
$$(C_b(X), C^\beta_b(X))_{\alpha, \infty} \subset \left\{ \begin{array}{ll}
C^{\alpha \beta}_b(X), &  \alpha \beta \notin \N, 
\\
\\
Z^{\alpha \beta}_b(X), &  \alpha \beta \in \N, 
\end{array}\right. $$
valid for every Banach space $X$ and for every $\beta >0$, $\alpha\in (0, 1)$.  

From the representation formula \eqref{P_t} we see that $P_t$ is a contraction semigroup in all spaces $C^{\beta}_b(X)$, and also in all Zygmund spaces. Therefore  \eqref{inclusione} yields that if $t\mapsto P_tf(x)$ is $\alpha$-H\"older continuous uniformly with respect to $x$, then $x\mapsto P_tf(x)$ belongs to $C^{\alpha /\theta}_b(X)$ (or to $Z^{\alpha /\theta}_b(X)$), uniformly with respect to $t$. 

So, we have an extension  to the present general setting of the heuristic principle ``time regularity implies space regularity" that holds for semigroups associated to uniformly elliptic operators  with bounded regular coefficients in $\R^N$.  
However the converse does not hold; in other words  the embeddings in \eqref{inclusione} cannot be replaced by equalities, in general. Under a   further assumption on the moments of the measures $\mu_t$, namely
\begin{equation}
\label{caratt_interpIntro}
\exists\, \gamma , \,C>0:\quad \int_X \|x\|^\gamma \,\mu_t(dx) \leq Ct^{\gamma \theta}, \quad 0<t\leq 1, 
\end{equation}
we prove that  for every $\alpha\in (0, 1\wedge \theta)\cap (0, \gamma \theta]$ we have
\begin{equation}
(C_b(X), D(L))_{\alpha, \infty} = C^{\alpha /\theta}_b(X) \cap Y_\alpha , 
\label{main}
\end{equation}
with equivalence of the respective norms, where
$$Y_\alpha :=\left\{ f\in C_b(X):\; [f]_{Y_\alpha} := \sup_{t > 0} \frac{\|f(T_t \cdot) -f\|_{\infty}}{t^{\alpha}} <+\infty \right\}, \quad \|f\|_{Y_\alpha} = \|f\|_{\infty} +  [f]_{Y_\alpha}. $$
If in addition all  the measures $\mu_t$ are centered, namely $\mu_t(B) = \mu_t(-B)$ for every Borel set $B$, and $\theta <1$, the above equivalence holds also for all  $\alpha \in (\theta, 2\theta) \cap (0, \gamma \theta]$.    
 
Both in finite and in infinite dimension,  popular examples of generalized Mehler semigroups are the Ornstein-Uhlenbeck semigroups (\S \ref{Sect.OU}), in which case 
the measures $\mu_t$ are Gaussian. We recall that if $X=\R^N$, given any matrices $A$, $Q=Q^*\geq 0$, and setting $Q_t = \int_0^t e^{sA}Qe^{sA^*}\,ds$, the relevant Ornstein-Uhlenbeck semigroup $P_t$ is defined by \eqref{P_t} where  $T_t= e^{tA}$, and $\mu_t$ is the Gaussian measure with mean $0$ and covariance $Q_t$. The generator  $L$ is a realization of the operator $\mathscr Lu (x) = $ Tr $(QD^2u(x))/2 + \langle Ax, \nabla u(x)\rangle$. If in addition $Q>0$, estimates \eqref{essenziali} holds with $\theta =1/2$ and \eqref{caratt_interpIntro} holds for any $\gamma >0$. The above characterization of $(C_b(\R^N), D(L))_{\alpha, \infty} $ was already proved in \cite{DPL}; in the subsequent paper \cite{ABP} it was re-discovered that $t\mapsto P_tf$ is $\alpha$-H\"older continuous with values in $C_b(\R^N)$  for $\alpha <1/2$ iff $f\in C^{2\alpha}_b(\R^N) \cap Y_\alpha$. A characterization of $(C_b(\R^N), D(L))_{\alpha, \infty} $ is available also in the case that Det$\,Q =0$ but $\mathscr L$ is hypoelliptic  (\cite{L0}). 

Still for Ornstein-Uhlenbeck semigroups, in infinite dimension
sufficient conditions for Hypothesis \ref{Hyp:LR}  to hold are well known in the case that $X$ is a separable Hilbert space; for any 
 $\theta \geq 1/2$ there are   examples such that \eqref{ipotesi_intro} is satisfied (e.g., \cite{C,DPZbrutto,ABP}).   Instead, it is not clear whether \eqref{caratt_interpIntro} holds for some $\gamma >0$. Therefore, we can prove only the embeddings \eqref{inclusione}. 

Nontrivial new examples of generators of generalized Mehler semigroups satisfying both Hypothesis \ref{Hyp:LR} and \eqref{caratt_interpIntro}
are Ornstein-Uhlenbeck operators with fractional diffusion in finite dimension (\S \ref{Sect.OU_diff_fraz}), such as 
$$(\mathscr L u)(x) =  \frac{1}{2}( {\text Tr}^s(QD^2u) )(x)  - \langle Ax, \nabla u(x)\rangle , \quad x\in \R^N, $$
with $s\in (0,1)$ and $Q$, $A$ matrices such that $Q=Q^*>0$. Here Tr$^s(QD^2)$ is the pseudo-differential operator with symbol $- \langle Q\xi, \xi\rangle^s$. In this case \eqref{ipotesi_intro} holds with $\theta = 1/(2s)$, the measures $\mu_t$ are absolutely continuous with respect to the Lebesgue measure, \eqref{caratt_interpIntro} holds for every $\gamma < 2s$, and we obtain
$$(C_b(X), D(L))_{\alpha, \infty} = C^{2s \alpha}_b(\R^N) \cap Y_\alpha , \quad \alpha \in (0,1), \; \alpha \neq 1/(2s). $$
%

\section{Interpolation in spaces of continuous and bounded functions}

\subsection{Function spaces}

Let $X$, $Y$ be Banach spaces. 

By $\mathcal B_b(X; Y)$ and $C_b(X; Y)$ we denote the space of all bounded Borel measurable (resp. bounded  continuous)   functions $F:X\mapsto Y$, endowed with the sup norm $\|F \|_{\infty}:= \sup_{x\in X} \|F(x)\|_Y$. 
If $X=\R$ we set $\mathcal B_b(X; \R) := \mathcal B_b(X)$ and  $C_b(X; \R):= C_b(X)$. 

For $\alpha\in (0, 1)$ we denote by $C^{\alpha}_b(X; Y)$ the space of the bounded and $\alpha$-H\"older continuous functions $F:X\mapsto Y$, 
endowed with the H\"older norm $\|F \|_{C^{\alpha}_b(X; Y)} = \|F\|_{\infty} + [F]_{C^{\alpha}(X; Y)}$, where 
$[F]_{C^{\alpha}(X; Y)}:= \sup_{x\in X, \, h\in X\setminus \{0\}} \|F(x+h)- F(x)\|_Y\|h\|_{X}^{-\alpha}$. 

For $\alpha =1$ we will not consider the Lipschitz condition, but   a weaker one. We set 
%
%
$$Z^1_b(X, Y) := 
 \bigg\{F\in C_b(X; Y):  \, [F]_{Z^1(X, Y)} := \sup_{x, h\in X, \,h\neq 0} \frac{\| F(x+2h)-2F(x+h) + F(x)\|_Y}{ \|h\|_{X}} <+\infty \bigg\}.  $$
The space $Z^1(X; Y)$ is called {\em Zygmund space}, and it is endowed with the norm
$$\|F\|_{Z^1_b(X, Y)} :=  \sup_{x\in X}\|F(x)\|_Y + [F]_{Z^1(X, Y)} .  $$

For every $n\in \N$ we denote by $C^k_b(X)$ the space of the bounded and $n$ times continuously Fr\'echet differentiable functions $f:X\mapsto \R$ with bounded Fr\'echet derivatives  up to the order $k$.  Its norm is
\begin{equation}
\label{normaC^k}
\|f\|_{C^{k}_{b}(X)} := \|f\|_{\infty} + \sum_{j=1}^k \sup_{x\in X} \|D^jf(x)\|_{\mathcal L^j(X)}, 
\end{equation}
where $\mathcal L^j(X)$ is the space of the $k$-linear continuous functions from $X^j$ to $\R$, endowed with the norm  $\|T\|_{ {\mathcal L}^{j}(X)} := \sup\bigg\{ \frac{|T(h_1, \ldots, h_j)|}{\|h_1\|\cdots \|h_j\|}: \; h_i\in X\setminus \{0\} \bigg\}. $

For $\sigma\in (0, 1)$ and $k\in \N$ we set
\begin{equation}
\label{def:higherHolder}
\begin{array}{c}
C_b^{\sigma +k}(X) := \{ f\in C_b^k(X): \: D^kf \in C^{\sigma}(X, \mathcal L^k(X))\}, 
\\
\\
\|f\|_{C^{\sigma +k}_b(X)} := \|f\|_{C_b^{k}(X)} + [D^kf]_{C^{\sigma}(X, \mathcal L^k(X))},
\end{array}
\end{equation}
and  for $k\in \N$, $k\geq 2$, the higher order Zygmund spaces are defined by 
\begin{equation}
\label{def:higherZygmund}
\begin{array}{c}
Z^{k}_b(X) := \{ f\in C^{k-1}_b(X): \: D^{k-1}f \in Z^{1}(X, \mathcal L^{k-1}(X))\}, 
\\
\\
\|f\|_{Z^{k}_b(X)} := \|f\|_{C^{k-1}_b(X)} + [D^{k-1}f]_{Z^{1}(X, \mathcal L^{k-1}(X))}.
\end{array}
\end{equation}
%

\subsection{Interpolation}

We recall that if $E$, $F$ are   Banach spaces and $F$ is  continuously embedded in $E$, the real interpolation space $(E,F)_{\alpha, \infty}$ is defined by 
$$(E,F)_{\alpha, \infty} = \{ f\in E:\; \|f\|_{(E,F)_{\alpha, \infty}} := \sup_{\xi>0} \xi^{-\alpha} K(\xi, f, E, F)<+\infty \}$$
where
$$ K(\xi, f, E, F) := \inf \{ \|a\|_E + \xi\|b\|_F:\; f= a+b, \; a\in E, \; b\in F \}, \quad \xi>0, \; f\in E. $$
There are several equivalent characterizations of real interpolation spaces. In the following we shall use the next one (e.g., \cite[Thm. 1.5.3]{T}). 

\begin{Proposition}
\label{medie}
$(E,F)_{\alpha, \infty} $ consists of the elements $f\in E$ such that $f= \int_0^{\infty} \frac{u(t)}{t} \,dt$, where $u\in C((0, +\infty);F)$ is such that 
 $t\mapsto   t^{1-\alpha}u(t)\in L^{\infty}((0, +\infty); F)$ and $t\mapsto  t^{-\alpha}u(t)\in L^{\infty}((0, +\infty); E)$. 
The norm 
$$f\mapsto \inf\left\{ \esssup_{t>0} \|t^{1-\alpha}u(t)\|_F + \esssup_{t>0}\|  t^{-\alpha}u(t)\|_E:\; f= \int_0^{\infty} \frac{u(t)}{t} \,dt \right\}$$
is equivalent to the norm of $(E,F)_{\alpha, \infty} $. 
\end{Proposition}

Moreover, given $\theta\in (0, 1)$ and three Banach spaces $G\subset F\subset E$, with continuous embeddings, we say that $F$ belongs to the class $J_\theta$ between $E$ and $G$ if there exists $C>0$ such that 
$$\|x\|_F\leq C\|x\|_G^\theta \|x\|_E^{1-\theta}, \quad x\in G. $$
In this case  the Reiteration Theorem (e.g., \cite[Thm. 1.10.2]{T}, \cite[Thm. 1.23(ii)]{L1}) yields, for every $\alpha\in (0, 1)$, $\alpha\neq \theta$, 
\begin{equation}
\label{reiterazione}
(E,G)_{\alpha , \infty} \subset \left\{ \begin{array}{ll} (E,F)_{\alpha /\theta, \infty}, & \text{if}\; \alpha <\theta, 
\\
\\
(F, G)_{ (\alpha-\theta)/(1-\theta), \infty}, & \text{if}\;  \alpha >\theta , 
\end{array}\right.
\end{equation}
with continuous embeddings. 

From now on we take $E= C_b(X)$, and $F= C^\beta_b(X)$, with $\beta >0$. 
Next statements are far from surprising. However, the known proofs  in the case $X=\R^N$ 
do not seem to be immediately extendable  to the infinite dimensional case, and therefore we provide   independent  proofs.

\begin{Lemma}
\label{C^k}
For every $k\in \N$, $k\geq 2$, there exists $C_k>0$ such that 
\begin{equation}
\label{Jintero}
\|D^hf\|_{L^\infty(X, \mathcal L^h(X))} \leq C_k  \|f\|_\infty^{1-h/k} \|D^kf\|_{L^\infty(X, \mathcal L^k(X))} , \quad h=1, \ldots, k-1, \; f\in C^k_b(X). 
\end{equation}
Similarly, for every noninteger $\beta >1$, $\beta= k+\sigma $ with $k\in \N$ and $\sigma\in (0, 1)$ there exists $C_\beta>0$ such that 
\begin{equation}
\label{Jnonintero}
\|D^hf\|_{L^\infty(X, \mathcal L^h(X))} \leq C_\beta \|f\|_\infty^{1-h/\beta} [D^kf]_{C^\sigma (X, \mathcal L^k(X))} , \quad h=1, \ldots, k, \; f\in C^\beta_b(X). 
\end{equation}
%
\end{Lemma}
\begin{proof}
Let us prove a basic estimate, for $f\in C^{1+\sigma}_b(X)$ with $\sigma \in (0, 1]$. 
For every $x\in X$, $h\in X$ with $\|h\|=1$ and $t>0$,   we have 
 $$\begin{array}{lll}
 | Df(x)(h) | &  \leq  & \ds \left| Df(x)(h) - \frac{ f(x+th) - f(x)}{t}\right| +  \left|\frac{ f(x+th) - f(x)}{t}\right|
 \\
 \\
 & \leq & \ds \frac{1}{t}  \int_0^t |Df(x)(h) - Df(x+ \tau h)(h)|\,d\tau +   \frac{2}{t}  \|f\|_{\infty}
 \\
 \\ 
  & \leq & \ds  [[Df]]_{\sigma } \frac{t^\sigma }{1+\sigma} +   \frac{2}{t} \|f\|_{\infty}, 
 \end{array}$$
where $ [[Df]]_{\sigma }:=  [Df]_{C^{\sigma}(X, X^*)}$ if $\sigma \in (0, 1)$, $ [[Df]]_{\sigma }:= \|D^2f\|_{L^\infty (X, \mathcal L^2(X))}$ if $\sigma =1$. 
In both cases,  for every $x\in X$, 
$$\|Df(x)\|_{X^*} \leq  [[Df]]_{\sigma } \frac{t^\sigma }{1+\sigma} +   \frac{2}{t} \|f\|_{\infty}, \quad  t>0. $$
Taking the minimum of the right hand side over $t>0$ (or else, choosing $t=  (\|f\|_{\infty}/ [[Df]]_{\sigma })^{1/(1+\sigma)}$),  we see that there exists $C $ such that 
\begin{equation}
\label{J1} \|Df(x)\|_{X^*} \leq C(  [[Df]]_{\sigma }^{1/(\sigma +1)} (\|f\|_{\infty})^{1-1/(\sigma +1)}, \quad x\in X, 
\end{equation}
which proves \eqref{Jintero} for $k=2$ and  \eqref{Jnonintero} for $\beta \in (1,2)$.  

Estimate \eqref{J1} is readily extended as follows: for every $k\in \N$,  and $\sigma \in (0, 1)$ 
 there exists $K_{k,\sigma} >0$ such that 
\begin{equation}
\label{J_k}
\|D^kf\|_{L^\infty (X, \mathcal L^k(X))} \leq   K_{k,\sigma} (\|D^{k-1}f\|_{L^\infty (X, \mathcal L^{k-1}(X))}^{1-1/(\sigma +1)} [D^kf]_{C^{ \sigma } (X, \mathcal L^k(X)) )})^{1/(\sigma +1)} , \quad f\in C^{k+\sigma}_b(X), 
\end{equation}
while for $\sigma =1$  there exists $K_{k} >0$ such that 
\begin{equation}
\label{J_kintero}
\|D^kf\|_{L^\infty (X, \mathcal L^k(X))} \leq   K_k  (\|D^{k-1}f\|_{L^\infty (X, \mathcal L^{k-1}(X))}^{1/2} \|D^{k+1}\|_{L^\infty (X, \mathcal L^{k+1}(X))})^{1/2} , \quad f\in C^{k+1}_b(X), 
\end{equation}
(for $k\geq 2$  it is sufficient to replace $f$ by $D^{k-1}f$ and argue as above).

Now we are ready to prove \eqref{Jintero}, by recurrence. We just proved that  \eqref{Jintero} holds for $k=2$. Assume that  \eqref{Jintero} holds for some $k\geq 2$. Using first \eqref{J_kintero} and then the recurrence assumption, for every $f\in C^{k+1}_b(X)$ we have
$$\begin{array}{lll}\|D^kf\|_{L^\infty (X, \mathcal L^k(X))}&  \leq  & K_k \|D^{k-1}f\|_{L^\infty (X, \mathcal L^{k-1}(X))}^{1/2}  \|D^{k+1}f\|_{L^\infty (X, \mathcal L^{k+1}(X))}^{1/2} 
\\
\\
& \leq & K_k  (C_{k}  \|f\|_{\infty} ^{1/k} 
\|D^kf\|_{L^\infty (X, \mathcal L^k(X))}^{(k-1)/k})^{1/2}  \|D^{k+1}f\|_{L^\infty (X, \mathcal L^{k+1}(X))}^{1/2} 
\\
\\
& = &  K_k  C_{k}^{1/2}   \|f\|_{\infty} ^{1/(2k)} \|D^kf\|_{L^\infty (X, \mathcal L^k(X))}^{(k-1)/(2k)}  \|D^{k+1}f\|_{L^\infty (X, \mathcal L^{k+1}(X))}^{1/2} 
\end{array}$$
so that 
\begin{equation}
\label{k-(k+1)}
\|D^kf\|_{L^\infty (X, \mathcal L^k(X))}\leq  ( K_k  C_{k}^{1/2})^{2k/(k+1)}   \|f\|_{\infty} ^{1/(k+1)}  \|D^{k+1}f\|_{L^\infty (X, \mathcal L^{k+1}(X))}^{k/(k+1)} . 
\end{equation}
For $h<k$ we use the recurrence assumption and then \eqref{k-(k+1)}, estimating
$$\begin{array}{lll}\|D^hf\|_{L^\infty (X, \mathcal L^h(X))} & \leq &  C_k    \|f\|_{\infty} ^{1-h/k} \|D^kf\|_{L^\infty (X, \mathcal L^k(X))}^{h/k}
\\
\\
& \leq & C_k( K_k  C_{k}^{1/2})^{2k/(k+1)}     \|f\|_{\infty} ^{1-h/k}  (\|f\|_{\infty} ^{1/(k+1)}  \|D^{k+1}f\|_{L^\infty (X, \mathcal L^{k+1}(X))}^{k/(k+1)})^{h/k}
\\
\\
& =  & C_k^{1+h/(k+1)} K_k^{2h/(k+1)}     \|f\|_{\infty} ^{1-h/(k+1)}  \|D^{k+1}f\|_{L^\infty (X, \mathcal L^{k+1}(X))}^{h/(k+1)}. 
\end{array}$$
Such estimates and \eqref{k-(k+1)} yield that   \eqref{Jintero} holds for $k+1$. 

Now we prove that   \eqref{Jnonintero} holds. Let $\beta= k+\sigma $ with $k\in \N$, $k\geq 2$,  and $\sigma\in (0, 1)$, and let $f\in C^{\beta}_b(X)$. Estimates \eqref{J_k} and   \eqref{Jintero} with $h= k-1$ 
yield
$$\begin{array}{lll}\|D^kf\|_{L^\infty (X, \mathcal L^k(X))} &\leq  & K_{k,\sigma} (\|D^{k-1}f\|_{L^\infty (X, \mathcal L^{k-1}(X))}^{1-1/(\sigma +1)} [D^kf]_{C^{ \sigma } (X, \mathcal L^k(X)) )})^{1/(\sigma +1)}
\\
\\
&  \leq &
 K_{k,\sigma} ( C_{k}\|f\|_{\infty}^{1/k} \|D^kf\|_{L^\infty (X, \mathcal L^k(X))}^{1-1/k})^{1-1/(\sigma +1)} [D^kf]_{C^{ \sigma } (X, \mathcal L^k(X))})^{1/(\sigma +1)} \end{array}$$
so that 
$$\|D^kf\|_{L^\infty (X, \mathcal L^k(X))} \leq ( K_{k,\sigma}  C_{k}^{1-1/(\sigma +1)})^{k(\sigma +1)/\beta} \|f\|_{\infty}^{\sigma/\beta}  [D^kf]_{C^{ \sigma } (X, \mathcal L^k(X))}^{k/\beta}$$
which gives \eqref{Jnonintero} with $h=k$. For $h<k$ we use \eqref{Jintero} and the above estimate, to get 
$$\begin{array}{lll}\|D^hf\|_{L^\infty (X, \mathcal L^h(X))} & \leq & C_k \|f\|_{\infty}^{1-h/k} \|D^kf\|_{L^\infty (X, \mathcal L^k(X))}^{h/k}
\\
\\
& \leq & C_k ( K_{k,\sigma}  C_{k}^{1-1/(\sigma +1)})^{k(\sigma +1)/\beta}  \|f\|_{\infty}^{1-h/\beta}   [D^kf]_{C^{ \sigma } (X, \mathcal L^k(X))}^{h/\beta},\end{array}$$
and  \eqref{Jnonintero} is proved. 
 \end{proof}

\begin{Corollary}
\label{Cor:J}
For every $k\in \N$ the mapping $f\mapsto \|f\|_{\infty} + \|D^kf\|_{L^\infty (X, \mathcal L^k(X))}$ is a norm in $C^k_b(X)$, equivalent to the norm \eqref{normaC^k}, 
and for every non integer $\beta >0$, $\beta = k+\sigma$ with $k\in \N$, $ \sigma \in (0, 1)$,  the mapping $f\mapsto \|f\|_{\infty} + [D^kf]_{C^\sigma (X, \mathcal L^k(X))}$ is a norm in $C^\beta_b(X)$, equivalent to the norm \eqref{def:higherHolder}. 
Moreover, for every $\beta >1$ and for every $h\in \N$, $h<\beta$, the space $C^h_b(X)$ belongs to the class $J_{h/\beta}$ between $C_b(X)$ and $C^\beta_b(X)$. 
\end{Corollary}

\begin{Proposition}
\label{Prop:inclusione}
For every $\beta >0$ and $\gamma \in (0, 1)$ we have 
\begin{equation}
\label{embedding}
(C_b(X), C^{\beta}_b(X))_{\gamma, \infty} \subset \left\{ \begin{array}{ll} 
C^{\gamma\beta}_b(X), & \gamma\beta \notin \N, 
\\
\\
Z^{\gamma\beta}_b(X), & \gamma\beta \in \N, 
\end{array}\right. 
\end{equation}
with continuous embeddings. 
\end{Proposition}
\begin{proof}
By Proposition \ref{medie}, each  $f\in (C_b(X), C^{\beta}_b(X))_{\gamma, \infty}$ may be represented as $ f= \int_0^{\infty} \frac{u(t)}{t}\,dt$, where $u\in C((0, +\infty);C^{\beta}_b(X))$,  $t\mapsto t^{1-\gamma}u(t)\in L^{\infty}((0, +\infty); C^{\beta}_b(X))$,  and $t\mapsto t^{-\gamma}u(t)\in L^{\infty}((0, +\infty);$ $C_b(X))$. 
By the norm equivalence of Proposition  \ref{medie} there is $C= C(\gamma) >0$ such that 
\begin{equation}
\label{stima0,beta}
(i)\;\|u(t)\|_{C_b (X)} \leq Ct^{\gamma}\|f\|_{(C_b(X), C^{\beta}_b(X))_{\gamma, \infty}}, \;\;
(ii)\; \|u(t)\|_{C_b^\beta(X)} \leq Ct^{\gamma -1}\|f\|_{(C_b(X), C^{\beta}_b(X))_{\gamma, \infty}},
\quad
t>0, 
\end{equation}
and  by \eqref{stima0,beta} and Lemma \ref{C^k} there is $C = C(\beta, \gamma)$ such that for every $k\in \N$, $k<\beta $ we have
\begin{equation}
\label{stimak}
\sup_{x\in X}\|D^ku(t)(x)\|_{\mathcal L^k(X)} \leq Ct^{\gamma -k/\beta}\|f\|_{(C_b(X), C^{\beta}_b(X))_{\gamma, \infty}}, \quad t>0. 
\end{equation}
Having such estimates at our disposal, the proof is now similar to the ones of Theorems 3.8(i) and 3.9(i) of \cite{LR}. 
 
First we consider the case that $\gamma\beta$ is not integer. Let $n =[\gamma\beta]$ 
be the integral part of $\gamma\beta$.

If  $n=0$, namely $\gamma\beta <1$, for every $x$, $y\in X$ and $\lambda >0$ we have
$$|f(x) - f(y)|   \leq   \int_0^\lambda  \frac{|u(t)(x) - u(t)(y)|}{t}\,dt + \int_\lambda^\infty  \frac{|u(t)(x) - u(t)(y)|}{t}\,dt $$
where by   \eqref{stima0,beta}(i)  the first integral is bounded by 
$$2\int_0^\lambda \frac{\|u(t)\|_\infty}{t} \, dt \leq 
2\int_0^\lambda t^{\gamma -1}dt \,C\|f\|_{(C_b(X), C^{\beta}_b(X))_{\gamma, \infty}} = \frac{2\lambda ^\gamma C}{\gamma}\|f\|_{(C_b(X), C^{\beta}_b(X))_{\gamma, \infty}} . $$
If $\beta \leq 1$,  by   \eqref{stima0,beta}(ii) the second integral is bounded by 
$$\begin{array}{l} 
\ds \int_\lambda^\infty \frac{\|u(t)\|_{C^\beta_b(X)}}{t}\,dt \;\|x-y\|^\beta \leq 
\int_\lambda^\infty  t^{\gamma -2}\,dt \, C\| f\|_{(C_b(X), C^{\beta}_b(X))_{\gamma, \infty}}  \|x-y\|^\beta
\\
\\
\ds = \frac{\lambda^{\gamma -1}C}{1-\gamma} \|f\|_{(C_b(X), C^{\beta}_b(X))_{\gamma, \infty}} \|x-y\|^\beta; \end{array}$$
while  if $\beta >1$, by \eqref{stimak} with $k=1$ the second integral is bounded by 
 $$\begin{array}{l} 
\ds   \int_\lambda^\infty   \frac{\|u(t)\|_{C^1_b(X)}}{t}\,dt \;\|x-y\|  \leq 
\int_\lambda^\infty  t^{\gamma -1/\beta -1}\,dt \, C\| f\|_{(C_b(X), C^{\beta}_b(X))_{\gamma, \infty}} \|x-y\|
\\
\\
\ds = \frac{\lambda^{\gamma -1/\beta}C}{1/\beta -\gamma} \|f\|_{(C_b(X), C^{\beta}_b(X))_{\gamma, \infty}} \|x-y\|. \end{array}$$
In both cases, taking $\lambda = \|x-y\|^\beta$ we get 
$$|f(x) - f(y)|   \leq  K \|x-y\|^{\gamma \beta}, $$
with $K>0$ independent of $f$, $x$, $y$. Therefore $f\in C^{\gamma\beta}_b(X)$ and  $[f]_{C^{\gamma\beta}(X)} \leq K\|f\|_{(C_b(X), C^{\beta}_b(X))_{\gamma, \infty}}$. 

Let now $n\in \N$. In this case, first we show that $f\in C^n_b(X)$, splitting $f = \int_0^\lambda  \frac{u(t)}{t}\,dt + \int_\lambda^{\infty} \frac{u(t)}{t}\,dt$ for every $\lambda >0$. The first integral has values in $C^n_b(X)$ since   $\gamma >n\beta$, so that 
$t\mapsto u(t)/t \in L^1((0, \lambda);C^n_b(X))$ by \eqref{stimak}. By \eqref{stima0,beta}(ii), the second integral has values in $ C^{\beta}_b(X)$  and therefore in $ C^{n}_b(X)$.  In both cases, $D^n$ commutes with the integral. 

Now we prove that $D^nf \in C^{\gamma\beta - n}_b(X, \mathcal L^n(X))$. As in the case $n=0$, for every $x$, $y\in X$ and $\lambda >0$ we split and estimate as follows, 
$$\begin{array}{lll}
\|(D^nf(x) - D^nf(y))\|_{\mathcal L^n(X)}  &  \leq  & \ds  \int_0^\lambda  \frac{\|D^nu(t)(x) - D^nu(t)(y)\|_{\mathcal L^n(X)}}{t}\,dt 
\\
\\
&&  \ds + \int_\lambda^\infty  \frac{\|D^nu(t)(x) - D^nu(t)(y)\|_{\mathcal L^n(X)}}{t}\,dt . \end{array}$$
In the first integral  we estimate $\|D^nu(t)(x) - D^nu(t)(y)\|_{\mathcal L^n(X)}     \leq 2 \|D^nu(t)\|_{L^\infty(X, \mathcal L^n(X))}$ and we use \eqref{stimak} with $k=n$; in the second integral we estimate $\|D^nu(t)(x) - D^nu(t)(y)\|_{\mathcal L^n(X)}\leq [D^nu(t)]_{C^{\beta -n}(X, \mathcal L^n(X))}\|x-y\|^{\beta -n}$ and we use \eqref{stima0,beta}(ii) if $\beta -n <1$, we estimate $\|D^nu(t)(x) - D^nu(t)(y)\|_{\mathcal L^n(X)}\leq 2 \|D^{n+1}u(t)]_{L^\infty (X, \mathcal L^{n+1}(X))}\|x-y\| $ and we use  \eqref{stimak} with $k=n+1$ if $\beta -n \geq 1$, getting respectively
$$\begin{array}{l}\|(D^nf(x) - D^nf(y))\|_{\mathcal L^n(X)}  
\\
\\
 \leq   \ds \int_0^\lambda 2   t^{\gamma -n/\beta -1}dt\, C\|f\|_{(C_b(X), C^{\beta}_b(X))_{\gamma, \infty}}
+  \int_\lambda^\infty   \|x-y\|^{\beta -n}t^{\gamma -2}dt\, C\|f\|_{(C_b(X), C^{\beta}_b(X))_{\gamma, \infty}}
\\
\\
\ds = C\left( \frac{2\lambda^{\gamma -n/\beta}}{\gamma - n/\beta} + \frac{\|x-y\|^{\beta -n}}{(1-\gamma)\lambda^{1-\gamma}}\right)   \|f\|_{(C_b(X), C^{\beta}_b(X))_{\gamma, \infty}}  \end{array} $$
if $\beta -n <1$, and
$$\begin{array}{l}\|(D^nf(x) - D^nf(y))\|_{\mathcal L^n(X)}  
\\
\\
 \leq   \ds \int_0^\lambda 2  Ct^{\gamma -n/\beta -1}dt\, \|f\|_{(C_b(X), C^{\beta}_b(X))_{\gamma, \infty}}
+  \int_\lambda^\infty  C \|x-y\| t^{\gamma -(n+1)/\beta -1}dt\, \|f\|_{(C_b(X), C^{\beta}_b(X))_{\gamma, \infty}} 
\\
\\
\ds = C\left( \frac{2\lambda^{\gamma -n/\beta}}{\gamma - n/\beta} + \frac{\|x-y\|^{\beta -n}}{((n+1)/\beta -\gamma)\lambda^{(n+1)/\beta-\gamma}}\right)   \|f\|_{(C_b(X), C^{\beta}_b(X))_{\gamma, \infty}}  \end{array} $$
if $\beta -n \geq 1$ (recall that $n= [\gamma \beta]$, so that $(n+1)/\beta -\gamma >0$). In both cases, 
taking again   $\lambda = \|x-y\|^\beta$, we get $\|(D^nf(x) - D^nf(y))\|_{\mathcal L^n(X)}   \leq K \|x-y\|^{\beta\gamma  -n} \|f\|_{(C_b(X), C^{\beta}_b(X))_{\gamma, \infty}} $ for some $K>0$ independent of $f$, $x$, $y$. So, $f\in C^{\beta\gamma}_b(X)$, and  $(C_b(X), C^{\beta}_b(X))_{\gamma, \infty}$ is continuously embedded in $C^{\beta\gamma}_b(X)$ by such estimate and Corollary \ref{Cor:J}.

If $\gamma\beta =n \in \N$ the proof is similar. For $n=1$, for every $x$, $y\in X$ and $\lambda >0$,  we estimate
$$\begin{array}{l}
|f(x) + f(y) -2f((x+y)/2)|   \leq 
\\
\\
\leq  \ds \int_0^\lambda  \frac{|u(t)(x) +u(t)(y) -2u((x+y)/2) |}{t}\,dt + \int_\lambda^\infty  \frac{|u(t)(x) +u(t)(y) -2u((x+y)/2) |}{t}\,dt .\end{array}$$
In the first integral  we use  estimate \eqref{stima0,beta}(i); in the second integral we estimate 
$$|u(t)(x) +u(t)(y) -2u((x+y)/2) | \leq  [Du(t)]_{C^{\beta -1 }(X, X^*)}\|x-y\|^{\beta} \leq C t^{\gamma -1}\|f\|_{(C_b(X), C^{\beta}_b(X))_{\gamma, \infty}}\|x-y\|^{\beta}$$ by \eqref{stima0,beta}(ii) if $\beta \in (1, 2)$;     and 
$$|u(t)(x) +u(t)(y) -2u((x+y)/2) | \leq   \|D^2u\|_{L^\infty (X, \mathcal L^2(X))}\|x-y\|^2 \leq  C t^{\gamma -  2/\beta -1}\|f\|_{(C_b(X), C^{\beta}_b(X))_{\gamma, \infty}}\|x-y\|^{2}$$  if $\beta \geq 2$, by \eqref{stimak} with $k=2$. In the first case we get 
$$\begin{array}{l}
|f(x) + f(y) -2f((x+y)/2)|   \leq \ds  \left(4\int_0^\lambda t^{\gamma -1}dt +  \|x-y\|^\beta \int_\lambda^\infty  t^{\gamma -2}dt\right) C\|f\|_{(C_b(X), C^{\beta}_b(X))_{\gamma, \infty}} 
\\
\\
\ds = \left( \frac{4\lambda^\gamma}{\gamma} + \frac{  \|x-y\|^\beta}{( 1-\gamma)\lambda ^{ 1-\gamma }}\right) C\|f\|_{(C_b(X), C^{\beta}_b(X))_{\gamma, \infty}}, \end{array}$$
while in the second case we get 
$$\begin{array}{l} |f(x) + f(y) -2f((x+y)/2)|   \leq \ds  \left(4\int_0^\lambda t^{\gamma -1}dt +  \|x-y\|^2 \int_\lambda^\infty  t^{\gamma - 2/\beta -1}dt\right) C\|f\|_{(C_b(X), C^{\beta}_b(X))_{\gamma, \infty}}
\\
\\
\ds = \left( \frac{4\lambda^\gamma}{\gamma} + \frac{ \|x-y\|^2}{(2/\beta -\gamma )\lambda ^{2/\beta -\gamma }}\right) C\|f\|_{(C_b(X), C^{\beta}_b(X))_{\gamma, \infty}}. \end{array}. $$
In both cases choosing $\lambda =  \|x-y\|^\beta$ and recalling that $\gamma \beta =1$ we get $|f(x) + f(y) -2f((x+y)/2)|   \leq  K \|f\|_{(C_b(X), C^{\beta}_b(X))_{\gamma, \infty}}\|x-y\|$ for some $K$ independent of $f$, $x$, $y$, so that $f\in Z^1_b(X)$  and  $(C_b(X), C^{\beta}_b(X))_{1/\beta, \infty}$ is continuously embedded in $Z^{1}_b(X)$. 

If $n>1$ the procedure is similar, replacing $f$ by $D^{n-1}f$ and distinguishing the cases $\beta \in (n, n+1)$ and $\beta \geq n+1$. The details are left to the reader. 
\end{proof}

It would be desirable to have equivalences  instead of embeddings in \eqref{embedding}, which  is true in  the case $X= \R^N$. In the  infinite dimensional case the only characterization of this kind that we are aware of is $(BUC(X), BUC^1(X))_{\gamma, \infty}$ $ = C^\gamma_b(X)$ for every $\gamma \in (0, 1)$, proved in \cite{DPC} when $X$ is a separable Hilbert space. 

\section{Good semigroups in $C_b(X)$}
\label{Sect:good}

Although the aim of this paper is the study of generalized Mehler semigroups, we start with a few properties satisfied by a larger class of semigroups. Indeed in this section we consider semigroups of bounded operators $P_t$ in $C_b(X)$, such that 
$$\|P_t\|_{\mathcal L (C_b(X))} \leq M e^{\beta t}, \quad t>0, $$
for some $M>0$, $\beta \in \R$, and such that the function $(t,x)\mapsto P_tf(x)$ is continuous in $[0, +\infty)\times X$. 

Then there exists a closed operator $L :D(L)\subset C_b(X)\mapsto C_b(X)$, such that the resolvent set of $L$ contains $(\beta, +\infty)$, and  
 \begin{equation}
 \label{L1}
 R(\lambda, L)f (x) = \int_0^{\infty} e^{-\lambda t}P_tf(x)\,dt, \quad \lambda >\beta, \;f\in C_b(X), \; x\in X. 
  \end{equation}
The (easy) proof of this statement may be found e.g. in \cite[\S 2(c)]{LP}. From \eqref{L1} we get immediately
$$\| R(\lambda, L)\|_{\mathcal L(C_b(X))} \leq \frac{M}{\lambda - \beta}, \quad \lambda >\beta. $$

Some properties of strongly continuous semigroups are shared by such semigroups. The first one is in the next lemma. 

\begin{Lemma} 
\label{Le:midnight}
For every $s>0$ and $f\in D(L)$, 
$P_sf\in  D(L)$,  we have $P_s Lf = LP_sf$, and 
\begin{equation}
\label{midnight}
P_tf(x) -f (x) = \int_0^t LP_sf(x)\,ds = \int_0^t P_sLf(x)\,ds, \quad t>0, \; x\in X. 
\end{equation}
Consequently, for every   $x\in X$ the function $t\mapsto P_tf(x)$ is  differentiable at any $t\geq 0$ with derivative $P_tLf(x)$, and moreover
\begin{equation}
\label{Lip}
  \|P_tf -f\|_{\infty} \leq t \,M\max\{ e^{\beta}, 1\}  \|Lf\|_{\infty}  , \quad 0<t\leq 1. 
 \end{equation}
\end{Lemma}
\begin{proof}
By formula  \eqref{L1}, $P_s$ commutes with $R(\lambda, L)$ for every $\lambda >\beta$, because it commutes with the Riemann sums 
$\frac{1}{nT}\sum_{k=1}^n e^{-kT/n}P_{kT/n}f$, for every $T>0$ and $n\in \N$. 
Therefore it commutes with $L$ on $D(L)$. 
Fixed any  $f\in D(L)$ and $\lambda >\beta $ set $g= \lambda f - Lf$ so that $f= R(\lambda, L)g$. For every $t>0$ and $x\in X$ we have 
$$\begin{array}{lll}
P_tf(x) -f (x)  &= & \ds  \int_0^{\infty} e^{-\lambda s}P_{t+s}g(x)\,ds - \int_0^{\infty} e^{-\lambda s}P_{s}g(x)\,ds 
\\
\\
& =  & 
\ds \int_t^{\infty} e^{-\lambda (\sigma -t)}P_\sigma g(x)\,d\sigma - \int_0^{\infty} e^{-\lambda s}P_{s}g(x)\,ds 
\\
\\
& =  & 
\ds (e^{\lambda t}-1)
\int_0^{\infty} e^{-\lambda \sigma}P_\sigma g(x)\,d\sigma - \int_0^t e^{-\lambda \sigma}P_\sigma g(x)\,d\sigma
\\
\\
& =  & \ds  (e^{\lambda t}-1)f(x) - \int_0^t e^{-\lambda \sigma}P_\sigma (\lambda f - Lf)(x)\,d\sigma . 
\end{array}$$
Although the proof of this formula makes sense only for $\lambda >\beta$, the formula holds for every $\lambda\in \C$ because  the right-hand side is a holomorphic function of $\lambda$.  
Letting $\lambda\to 0$, we get $P_tf(x) -f (x) = \int_0^t  (P_\sigma   Lf)(x)\,d\sigma$ and \eqref{midnight} follows. 

Since $s\mapsto P_sLf(x)$ is continuous in $[0, +\infty)$, by \eqref{midnight} 
 $t\mapsto P_tf(x)$ is differentiable at any $t\geq 0$ with derivative  $ P_t Lf(x)$. Estimate \eqref{Lip} follows immediately from \eqref{midnight}. \end{proof}

\begin{Corollary}
\label{Cor:strong_cont}
For every $f\in C_b(X)$ we have 
$$\lim_{t\to 0} \|P_tf -f\|_{\infty} =0 \Longleftrightarrow \lim_{\lambda \to \infty} \|\lambda R(\lambda, L)f -f\|_{\infty} =0 \Longleftrightarrow f\in \overline{D(L)}. $$
\end{Corollary}
\begin{proof}
For every $f\in D(L)$ we have $\lim_{t\to 0} \|P_tf -f\|_{\infty} = \lim_{\lambda \to \infty} \|\lambda R(\lambda, L)f -f\|_{\infty} =0$, respectively by \eqref{Lip} and by the equality $ \lambda R(\lambda, L)f = R(\lambda, L)Lf + f$. Since there is $C>0$ such that $\|P_t-I\|_{\mathscr L(C_b(X))} \leq C$ for every $t\in (0, 1)$, $\|\lambda R(\lambda, L)- I\|_{\mathscr L (C_b(X))} \leq C$ for every $\lambda > \beta +1$, one gets as well $\lim_{t\to 0} \|P_tf -f\|_{\infty} = \lim_{\lambda \to \infty} \|\lambda R(\lambda, L)f -f\|_{\infty} =0$ for every $f\in  \overline{D(L)}$. 

Conversely, if $ \lim_{\lambda \to \infty} \|\lambda R(\lambda, L)f -f\|_{\infty} =0$ then $f\in \overline{D(L)}$ since $\lambda R(\lambda, L)f\in D(L)$ for every $\lambda >0$. If $\lim_{t\to 0} \|P_tf -f\|_{\infty} =0$, for every $\eps >0$ let $\delta >0$ be such that $\|P_tf -f\|_{\infty} \leq \eps$ for every $t\in [0, \delta]$. For each $\lambda >\beta $ and $x\in X$ we have
$$\lambda R(\lambda, L)f(x)  -f(x) = \left( \int_0^\delta + \int_\delta^{\infty}\right) \lambda e^{-\lambda t}(P_tf (x) - f(x))\,dt, $$
so that 
 $\|\lambda R(\lambda, L)f -f\|_{\infty}  \leq \delta \eps +   \sup_{t\geq \delta} \lambda e^{-\lambda t}(M e^{\beta t } +1)\|f\|_{\infty}$, 
and therefore $\lim_{\lambda\to \infty}\|\lambda R(\lambda, L)f -f\|_{\infty} =0$, which implies    $f\in \overline{D(L)}$. 
\end{proof}

Now we consider   H\"older continuity. 
It is well known that for every strongly continuous  semigroup $P_t$ in a Banach space $E$ and for every $f\in E$, $\alpha \in (0, 1)$,  the function $t\mapsto P_tf$ belongs to $C^{\alpha}([0, T]; E)$ for some/all $T \in (0, +\infty)$ if and only if $f\in (E, D(L))_{\alpha, \infty}$, where $L$ is the infinitesimal generator of $P_t$. 
This characterization goes back to the pioneering paper by Lions and Peetre \cite{LP} and may be found in (almost) any book about interpolation theory. See e.g. \cite[Sect. 1.13]{T}.  
It still holds for analytic semigroups not necessarily strongly continuous at $0$ (\cite[Sect. 2.2]{L}) and for some non analytic semigroups in spaces of continuous and bounded functions (\cite[Ch. 5]{L1}). In fact in the proof of the next Proposition \ref{Pr:LionsPeeetre} we use a result  from \cite{L1}. 
\begin{Proposition}
\label{Pr:LionsPeeetre}
If  $\beta \leq 0$, 
for every $f\in C_b(X)$ and $\alpha\in (0, 1)$ the following conditions are equivalent:
\begin{itemize}
\item[(i)] $[f]^{(1)}_{\alpha}:=\ds \sup_{ t\in (0, 1)} t^{-\alpha} \|P_t f-f\|_{\infty}  < +\infty$,
\\
\item[(ii)] $[f]^{(2)}_{\alpha}:= \ds \sup_{ t>0}  t^{-\alpha} \|P_t f-f\|_{\infty} <+\infty$,
\\
\item[(iii)] $[f]^{(3)}_{\alpha}:=\sup_{\lambda >0} \| \lambda ^{\alpha} LR(\lambda, L)f\|_{\infty} <+\infty$,
\\
\item[(iv)] $f\in (C_b(X), D(L))_{\alpha, \infty}$,
\end{itemize}
and  the norms $f\mapsto \|f\|_{\infty} + [f]^{(1)}_{\alpha} $, $f\mapsto \|f\|_{\infty} + [f]^{(2)}_{\alpha} $, $f\mapsto \|f\|_{\infty} + [f]^{(3)}_{\alpha} $ are equivalent to the norm of $(C_b(X), D(L))_{\alpha, \infty}$. 

If $\beta >0$, for every $f\in C_b(X)$ and $\alpha\in (0, 1)$ condition  (i) is equivalent to (iv), and the norm $f\mapsto \|f\|_{\infty} + [f]^{(1)}_{\alpha} $  is equivalent to the norm of $(C_b(X), D(L))_{\alpha, \infty}$. 
\end{Proposition}
\begin{proof}
Let $\beta \leq 0$. Since $\|P_t\|_{\mathcal L (C_b(X))} \leq M$ for every $t$, the equivalence between (i) and (ii) is obvious, as well as the inequalities $[f]^{(1)}_{\alpha} \leq [f]^{(2)}_{\alpha} \leq \max \{ [f]^{(1)}_{\alpha}, \,(M+1)\|f\|_{\infty} \}$ for every $f\in C_b(X)$ and $\alpha\in (0, 1)$. 

Let us prove that (ii) implies (iii). For every $\lambda >0$ we have $ LR(\lambda, L)f = \lambda R(\lambda, L)f -f$, so that for every $x\in X$ 
$$| \lambda ^{\alpha} LR(\lambda, L)f(x) | = \left| \int_0^{+\infty} \lambda ^{\alpha +1} t^{\alpha} e^{-\lambda t}\frac{P_tf(x) - f(x)}{t^{\alpha}}\,dt . \right|$$
Therefore, 
$$\|\lambda ^{\alpha} LR(\lambda, L)f\|_{\infty} \leq \Gamma(\alpha +1)[f]^{(1)}_{\alpha}$$
so that (iii) holds, and  $[f]^{(3)}_{\alpha}\leq \Gamma(\alpha +1)[f]^{(1)}_{\alpha}$. 

The equivalence between (iii) and (iv), as well as the relevant norms,  follows from \cite[Prop. 3.1]{L1}, where no assumption on the density of $D(L)$ was made. 

If (iv) holds, for every $t>0$ and for every decomposition $f= a+b$, with $a\in C_b(X)$, $b\in D(L)$, we have $ \|P_tb - b\|_{\infty}\leq t\|Lb\|_{\infty}$ by 
\eqref{Lip}, and therefore
$$\frac{\|P_tf - f\|_{\infty}}{t^\alpha} \leq \frac{\|P_ta - a\|_{\infty}}{t^\alpha} + \frac{\|P_tb - b\|_{\infty}}{t^\alpha} \leq  \frac{2\|  a\|_{\infty}}{t^\alpha} +
 \frac{t\|Lb\|_{\infty}}{t^\alpha} . $$
Taking the infimum over all the decompositions we get  
$$\frac{\|P_tf - f\|_{\infty}}{t^\alpha} \leq 2 t^{-\alpha} K(t,f,C_b(X), D(L)), \quad t>0, $$
so that (ii) holds, and $[f]^{(2)}_{\alpha}\leq 2 \|f\|_{(C_b(X), D(L))_{\alpha, \infty}}$.

If $\beta >0$   we consider  the semigroup $S_t: e^{-\beta t}P_t$, whose generator is $L_\beta:= L-\beta I :D(L)\mapsto C_b(X)$, 
and $R(\lambda, L_\beta) = R(\lambda + \beta, L)$ for $\lambda >0$. So, conditions (i) to (iv) are equivalent for $S_t$. 
Since the graph norm of $L_\beta$ is equivalent to the graph norm of $L$, 
we still obtain   (i) $\Longleftrightarrow$ (iv), with equivalence of the norms $f\mapsto \|f\|_{\infty} + [f]^{(1)}_{\alpha} $,   $f\mapsto \|f\|_{(C_b(X), D(L))_{\alpha, \infty}}$. 
\end{proof}

Using the semigroup law for $P_t$ yields that for every $T>0$ the function $t\mapsto  P_tf =:u(t)$ belongs to $C^\alpha([0,T]; C_b(X))$ if and only if $f\in (C_b(X), D(L))_{\alpha, \infty}$. If   $\beta \leq 0$, in such a case $u\in C^\alpha ([0, +\infty); C_b(X))$. We may go on, characterizing the initial functions $f$ such that $t\mapsto P_tf$ is more regular. 
  
\begin{Corollary}
For every $f\in C_b(X)$ and $\alpha\in (0, 1)$, the function $t\mapsto P_tf $ belongs to $C^1([0, +\infty);$
$C_b(X))$ if and only if $f\in D(L)$ and $Lf\in \overline{D(L)}$; it belongs to $C^{1+\alpha }([0, T];C_b(X))$ for some/every $T>0$ if and only if $f\in D(L)$ and $Lf\in (C_b(X), D(L))_{\alpha, \infty}$. 
\end{Corollary}
\begin{proof}
Set $u(t):= P_tf$. If $f\in D(L)$ and $Lf\in  \overline{D(L)} $ then $s\mapsto P_sLf$ belongs to $C([0, +\infty); C_b(X))$, so that $u$  belongs to $C^1([0, +\infty);C_b(X))$ by \eqref{midnight} and $u'(t)= P_tLf =Lu(t)$. If in addition $Lf\in (C_b(X), D(L))_{\alpha, \infty}$, then $u'\in C^{\alpha }([0, T];C_b(X))$ for every $T>0$ and therefore $u\in C^{1+\alpha }([0, T];C_b(X))$ for every $T>0$. 

Conversely, assume that $u$ is continuously differentiable. Then $u'(t) = \lim_{h\to 0^+}P_t(( P_hf - f)/h) = P_t u'(0)$. Since $u'$ is continuous, $u'(0)\in \overline{D(L)}$ by Corollary \ref{Cor:strong_cont}, and moreover integrating by parts in \eqref{L} we get, for every $\lambda >\beta$, 
$$\lambda R(\lambda, L)f = \int_0^{\infty} e^{-\lambda t} u'(t) dt + f = R(\lambda, L)u'(0) + f$$
which yields $f = R(\lambda, L)( \lambda f -  u'(0)) \in D(L)$ and $R(\lambda, L)Lf = \lambda R(\lambda, L)f  -f = R(\lambda, L)u'(0) $ so that $Lf=u'(0) \in 
 \overline{D(L)}$.  
 If in addition $u'(t) = P_tLf$ belongs to $C^\alpha ([0,T]; C_b(X))$ for some $T>0$, then $Lf\in (C_b(X), D(L))_{\alpha, \infty}$ by Proposition  \ref{Pr:LionsPeeetre} and therefore $u\in C^{1+\alpha }([0, T];C_b(X))$ for  every $T>0$. 
  \end{proof}

\section{Generalized Mehler semigroups}

Throughout this section  $P_t$ is the generalized Mehler semigroup defined by  \eqref{P_t},
where $T_t$ is any strongly continuous semigroup of bounded operators  on  $X$, and $\{\mu_t:\; t\geq 0\}$ is a family of Borel probability measures in $X$ such that $\mu_0= \delta_0$  , $t\mapsto \mu_t$ is weakly continuous in $[0, +\infty)$ and   \eqref{semigruppo} holds. 

Notice that 
\begin{equation}
\label{contraz}
\|P_t\|_{ \mathscr L(C_b(X))} = 1, \;\;t>0, \quad \|R(\lambda, L)\|_{ \mathscr L(C_b(X))} = \frac{1}{\lambda}, \;\;\lambda >0. 
\end{equation}
The inequalities $\leq$ are an immediate consequence of \eqref{P_t}, the equalities follow taking $f\equiv 1$. Another immediate consequence of the representation formula \eqref{P_t} is that $P_t$ preserves $C^k_b(X)$ for every $k\in \N$. In particular, 
$$DP_tf(x) = T_t^* \int_X Df(T_tx+y)\mu_t(dy), \quad f\in C^1_b(X), \;x\in X, $$
and therefore there are $M>0$, $\omega\in \R$ such that 
\begin{equation}
\label{C^1}
\|DP_tf \|_{L^{\infty} (X, X^*)} \leq Me^{\omega t}  \|D f \|_{ L^{\infty} (X, X^*)}, \quad t>0,\; f\in C^1_b(X). 
\end{equation}
Similarly, for every  $k\in \N$, 
\begin{equation}
\label{C^k}
\|D^kP_tf \|_{L^{\infty} (X, \mathscr L^k(X))} \leq Me^{k\omega t}  \|D^k f \|_{ L^{\infty} (X,  \mathscr L^k(X))}, \quad t>0,\; f\in C^k_b(X). 
\end{equation}

It is not hard to check that for every $f\in C_b(X)$ the function $(t,x)\mapsto P_tf(x)$ is continuous in $[0, +\infty)\times X$. See e.g. \cite[Lemma 2.1]{BRS}. Therefore the results of Section \ref{Sect:good} hold for $P_t$.

Notice that the semigroup $S_t$ defined by $S_tf(x) := f(T_tx)$ is itself a generalized Mehler semigroup, with $\mu_t = \delta_0$ for every $t$. 
We denote by $L_0$ its generator, and by $Y_0$ the subspace of strong continuity of $S_t$, namely
$$Y_0 := \{ f\in C_b(X):\; \lim_{t\to 0} \|f(T_t\cdot) - f\|_{\infty} =0\}. $$
For $\alpha\in (0, 1)$ we denote by $Y_\alpha$ the subspace of $\alpha$-H\"older continuity of $S_t$, namely
$$Y_\alpha  := \left\{ f\in C_b(X):\; [f]_{Y_\alpha} := \sup_{t>0} \frac{\|f(T_t\cdot ) - f\|_{\infty}}{ t^{\alpha}} <+\infty \right\}, $$
endowed with the norm
$$\|f\|_{Y_\alpha}:= \|f\|_{\infty} +  [f]_{Y_\alpha}. $$
By Proposition  \ref{Pr:LionsPeeetre} it coincides with the 
interpolation space $(C_b(X), D(L_0))_{\alpha, \infty} $, and the respective norms are equivalent. 

Notice that if $g\in R(\lambda, L)(C^1_b(X))$ for some $\lambda >0$ and $x\in D(A)$, integrating by parts in \eqref{L} we get $L_0g(x) = g'(x)(Ax)$, so that  $L_0$ is a realization of a drift operator. 

Going back to the general case, 
Corollary \ref{Cor:strong_cont} gives an abstract characterization of the subspace of strong continuity of $P_t$. Under some additional assumptions we prove an explicit characterization of it. 

\begin{Proposition}
\label{Pr:strong_cont}
Assume that $P_t$ maps $C_b(X)$ into $BUC(X)$, for every $t>0$. Then 
$$ \overline{D(L)} = BUC(X) \cap Y_0. $$
\end{Proposition}
\begin{proof}
Let $f\in  BUC(X)$. Fixed any $\eps >0$ let $r$ be such that $|f(z+y)-f(y)|\leq \eps$ for every $z\in X$ and $y\in B(0, r)$. 
Then for every $t>0$ and $x\in X$ we have
$$\begin{array}{lll}
|P_tf(x) - f(T_tx)| & \leq & \ds \int_X |f(T_tx+y) - f(T_tx)| \,\mu_t(dy) = 
\\
\\
& \leq & \ds \left( \int_{B(0, r)} +  \int_{X\setminus B(0, r)}\right) |f(T_tx+y) - f(T_tx)| \,\mu_t(dy)
\\
\\
&\leq & \eps + 2\|f\|_{\infty} \mu_t(X\setminus B(0, r)). 
\end{array}$$
Since $\mu_t$ weakly converges to $\delta_0$, $\lim_{t\to 0} \mu_t(X\setminus B(0, r))=0$. Therefore, for $t$
 small enough we have $|P_tf(x) - f(T_tx)|\leq 2\eps$, so that 
 \begin{equation}
 \label{intermediate}
 \lim_{t\to 0} \|P_tf - f(T_t\cdot)\|_{\infty} =0, \quad f\in BUC(X). 
 \end{equation}

Now, let $f\in  \overline{D(L)} $. Then $\|P_tf -f\|_{\infty} \to 0$ as $t\to 0$ by Corollary \ref{Cor:strong_cont} , and since $P_tf\in BUC(X)$ for $t>0$, also $f\in BUC(X)$. 
Moreover, $\|f(T_t\; \cdot )- f  \|_{\infty} \leq  \| f(T_t\;\cdot) -P_tf \|_{\infty} + \|P_tf -f\|_{\infty}$ vanishes as $t\to 0$ by \eqref{intermediate} and Corollary \ref{Cor:strong_cont}. 

Conversely, if $f\in BUC(X) \cap Y_0$ then $\lim_{t\to 0}\|P_tf -f\|_{\infty} \to 0$ by \eqref{intermediate}. 
\end{proof}

As for strong continuity, also for H\"older continuity we  give an explicit characterization under further assumptions. 
Preliminarly we  recall some definitions and properties of Borel measures in Banach spaces, taken from \cite{BogaDiff}. 

A Borel probability measure $\mu$  in $X$ is called {\em Fomin differentiable} along $v\in X$ if for every Borel set $B$ the incremental ratio 
$ (\mu (B+tv) - \mu(B))/t  $ has finite limit  as $t\to 0$. Such a limit is a signed measure, called  $d_v\mu(B)$; setting $\mu_v(B):= \mu(B+v)$ for every Borel set $B$  we have  $\lim_{t\to 0}\left\| \frac{\mu_{tv}- \mu}{t} - d_v\mu\right\| =0$,
where $\|\cdot\|$ denotes the total variation norm. 

Moreover, $d_v\mu$ is absolutely continuous with respect to $\mu$. The density $\beta^{\mu}_v\in L^1(X, \mu)$ is called {\em Fomin derivative}   of $\mu$ along $v$, and it satisfies
\begin{equation}
\label{Fomin}
\int_X \frac{\partial f}{\partial v}\,\mu(dx) = - \int_X \beta^{\mu}_vf\,\mu(dx), \quad f\in C^1_b(X). 
\end{equation}
Conversely, if \eqref{Fomin} holds for some $\beta^{\mu}_v\in L^1 (X, \mu)$ and for every $f\in C^1_b(X)$, then $\mu$  is Fomin differentiable along $v$ (\cite[Thm. 3.6.8]{BogaDiff}).

Under Hypothesis \ref{Hyp:LR} it was proved in \cite{LR} that $P_tf$ has Gateaux derivatives of any order, for every $t>0$ and for every $f\in C_b(X)$. 
To be more precise,  in \cite{LR} it was assumed that $X$ is separable; however this condition was used only to guarantee that every Borel measure in $X$ is Radon, in order to apply the results on differentiable measures contained in \cite[\S 3.3]{BogaDiff}.   
Here we prove better smoothing properties, showing that $P_t$ is strong-Feller (namely, that it maps $B_b(X)$ into $C_b(X)$), and that all the Gateaux derivatives of $P_tf$ are in fact Fr\'echet derivatives.

\begin{Lemma}
\label{strong-Feller}
Let  Hypothesis \ref{Hyp:LR} hold. For each $f\in B_b(X)$, $P_tf\in C^k_b(X)$ for every $t>0$ and $k\in \N$. Moreover for every $k\in \N$, the function $(t,x)\mapsto P_tf(x)$ belongs to $C((0, +\infty)\times X)$ if $f\in B_b(X)$, and to  $C([0, +\infty)\times X)$ if $f\in C_b(X)$. 
\end{Lemma}
\begin{proof} 
Fix $f\in B_b(X)$. For every  $x$, $h\in X$ we have 
$$P_tf(x+h)  = \int_X f(T_tx + T_th +y)\mu_t(dy) =  \int_X f(T_tx + z )(\mu_t)_{T_th}(dz), $$
so that 
$$|P_tf(x+h) - P_tf(x)| = \left| \int_X f(T_tx + z )((\mu_t)_{T_th}(dz)-\mu_t(dz)\right| \leq \|f\|_{\infty} \| (\mu_t)_{T_th} -\mu_t\|. $$
By \cite[Thm. 3.3.7(i)]{BogaDiff} we have $ \| (\mu_t)_{T_th} -\mu_t\| \leq  \| \beta_{t,h}\|_{L^1(X, \mu_t)} $, so that 
$$|P_tf(x+h) - P_tf(x)| \leq \frac{C e^{\omega t}}{t^\theta }\|h\|, \quad t>0, \; h\in X. $$
Therefore, $P_tf$ is Lipschitz continuous for $t>0$. Let us prove that in fact it belongs to $C^k_b(X)$ for every $k\in \N$. 

In \cite{LR} we remarked that if $g\in BUC(X)$, then $P_tg$ is Fr\'echet differentiable at every $x\in X$, since the Gateaux derivative is continuous with values in $X^*$. Now we extend this remark to derivatives of arbitrary order $n$, through the representation formula for the $n$-th order Gateaux derivative $D^n_GP_tg$ of \cite[Prop. 3.3(ii)]{LR}, 
$$\begin{array}{lll}
 D^n_GP_tg(x)(h_1,\ldots,h_n)  & = & \ds (-1)^n\int_X\cdot\cdot\cdot\int_X g(T_tx+T_{\frac{n-1}{n}t}y_1+\cdot\cdot\cdot+T_{\frac t n}y_{n-1}+y_n)\cdot 
 \\
 \\
& &\quad \quad  \quad\cdot  \beta_{ t/n, h_n}(y_n)\cdot\cdot\cdot\beta_{t/n,h_1} (y_1)\mu_{t/n}(dy_n)\cdot\cdot\cdot\mu_{t/n}(dy_1)  . 
\end{array}$$
So, for $x$, $x_0\in X$  we have 
$$\begin{array}{l}
(D^n_GP_tg(x) - D^n_GP_tg(x_0))(h_1,\ldots,h_n)| \leq 
\\
\\
\ds \int_X\cdot\cdot\cdot\int_X |g(T_tx+T_{\frac{n-1}{n}t}y_1+\cdot\cdot\cdot+T_{\frac t n}y_{n-1}+y_n) - g(T_tx_0+T_{\frac{n-1}{n}t}y_1+\cdot\cdot\cdot+T_{\frac t n}y_{n-1}+y_n)| \cdot
\\
\\
\quad \quad  \quad\cdot | \beta_{ t/n, h_n}(y_n)| \cdot\cdot\cdot |\beta_{t/n,h_1} (y_1)|\mu_{t/n}(dy_n)\cdot\cdot\cdot\mu_{t/n}(dy_1)
\\
\\
\leq (Ce^{\omega t/n} n^\theta t^{-\theta})^n \|h_1\| \cdot\cdot\cdot  \|h_n\|\, \sup_{z\in X} |g(T_tx +z) - g(T_tx_0 +z)|
\end{array}$$
By the uniform continuity of $g$, $D^n_GP_tg$ is continuous at $x_0$ with values in $\mathcal L^n(X)$, and therefore it is a Fr\'echet derivative. 

Now, let $f\in B_b(X)$. For $t>0$, the function $g= P_{t/2}f$ belongs to $BUC(X)$, by the first part of the proof, and therefore $P_tf = P_{t/2}g$ has Fr\'echet derivatives of any order. 

The continuity of the function $[0, +\infty)\times X\mapsto \R$, $(t,x)\mapsto P_tf(x)$  if $f\in C_b(X)$ was proved in \cite[Lemma 2.1]{BRS}. 
If $f\in  B_b(X)$, $P_{t_0}f \in C_b(X)$ for every $t_0>0$, and therefore $(t,x)\mapsto P_tf(x) = (P_{t-t_0}P_{t_0}f )(x)$ is continuous in $[t_0, +\infty)\times X$. 
Since $t_0$ is arbitrary, $(t,x)\mapsto P_tf(x) \in C((0, +\infty)\times X)$. 
\end{proof}

By Lemma \ref{strong-Feller}, the pointwise estimates on the Gateaux derivatives of $P_tf$ in \cite{LR} are in fact estimates on the Fr\'echet derivatives. They amount to 
\begin{equation}
\label{nth-derivatives1}
\|D^n P_tf(x)\|_{\mathcal L^n(X)} \leq K_n \max\{ 1,  t^{-n\theta} \} \|f\|_{\infty}, \quad t>0, \; x\in X, \; f\in B_b(X), 
\end{equation}
for some $K_n>0$ (\cite[Prop.  3.3(ii)]{LR}). By Lemma 2.1(i) of \cite{LR}, with $Y= \mathcal L^n(X)$, we obtain estimates for H\"older seminorms, namely for every $n\in \N$, $\alpha \in (0, 1) $ there exists $K_{n,\alpha}>0$ such that 
\begin{equation}
\label{nth-derivativesHolder}
[D^n P_tf]_{C^\alpha (X,  \mathcal L^n(X))} \leq K_{n,\alpha} \max\{ 1,  t^{-(n+\alpha)\theta} \} \|f\|_{\infty}, \quad t>0, \; x\in X, \; f\in B_b(X). 
\end{equation}

Moreover, Lemma \ref{strong-Feller} and Theorem 3.8(i) of \cite{LR} yield
\begin{equation}
\label{D(L)}
 D(L)  \subset \left\{ \begin{array}{ll}C^{1/\theta}_b(X), & 1/\theta \notin \N, 
 \\
 \\
 Z^{1/\theta}_b(X), & 1/\theta \in \N, 
 \end{array}\right.
 \end{equation}
with continuous embedding.

\begin{Proposition}
\label{pr:basic}
If Hypothesis  \ref{Hyp:LR}  holds, we have
 $$(C_b(X), D(L))_{\alpha, \infty} \subset \left\{ \begin{array}{ll}C^{\alpha/\theta}_b(X), & \alpha/\theta \notin \N, 
 \\
 \\
 Z^{\alpha/\theta}_b(X), & \alpha/\theta \in \N, 
 \end{array}\right. $$
 with continuous embedding.
\end{Proposition}
 \begin{proof} If $ 1/\theta \notin \N$, the statement is a direct consequence of the embedding \eqref{D(L)} and of Proposition \ref{Prop:inclusione}, recalling that if $F_1\subset F_2\subset E$ with continuous embeddings then $(E, F_1)_{\alpha, \infty} \subset (E, F_2)_{\alpha, \infty}$ with continuous embedding, for every $\alpha \in (0, 1)$. 
 
If $1/\theta\in \N$, we show that for every non integer $\beta\in (0,  1/\theta)$  the space $C^{\beta}_b(X)$ belongs to the class $J_{\theta \beta}$ between $C_b(X)$ and $D(L)$. 

Let $\beta = k + \sigma$, with $k\in \N \cup \{0\}$ and $\sigma \in (0, 1)$. 
 
For every $\varphi \in D(L)\setminus\{0\}$, set  $f = \lambda \varphi - (L- I)\varphi $,  so that  $\varphi (x) = \int_0^{\infty} e^{-(\lambda +1 )t}P_tf (x)dt $ for every $x\in X$. 
By proposition 3.7 of \cite{LR} we have  $D^j\varphi (x)(h_1, \ldots, h_j) = \int_0^{\infty} e^{-(\lambda +1 )t}(D^hP_tf )(x)(h_1, \ldots, h_j)dt $ for every 
$j \leq k$, $(h_1, \ldots, h_j) \in X^j$ and $x\in X$. By estimates \eqref{nth-derivatives1}, for every $j\leq k$ we have
$$\begin{array}{lll}
|D^j\varphi (x)(h_1, \ldots, h_j) | & \leq & \ds  \left( \int_0^1 \frac{e^{-(\lambda +1)t}}{t^{\theta j}} dt + \int_1^\infty e^{-(\lambda +1)t}dt\right) K_j \| f\|_\infty
\prod_{l=1}^j\|h_l\|
\\
\\
& \leq & \ds \left( \frac{ \Gamma (1-\theta j)}{(\lambda +1)^{(1-\theta j)}} + \frac{1}{\lambda +1}\right) K_j \| f\|_\infty \prod_{l=1}^j\|h_l\|
\end{array}$$
so that there exists $C_1>0$ such that $\|D^j\varphi\|_{L^\infty (X, \mathcal L^j(X))} \leq C_1  \lambda^{\theta\beta -1}\| f\|_\infty$ for every $j\leq k$. Similarly, using \eqref{nth-derivativesHolder} with $n=k$ and $\alpha =\sigma$, we get $[D^k\varphi]_{C^\sigma(X, \mathcal L^k(X))} \leq C_2 \lambda^{\theta\beta -1}\| f\|_\infty$ for some $C_2>0$. Summing up, there exists $C>0$ such that 
$$\| \varphi\|_{C^\beta _b(X)} \leq C \lambda^{\theta\beta -1}\| f\|_\infty \leq C (  \lambda^{\theta\beta }\|\varphi\|_\infty +  \lambda^{\theta\beta -1}\|(L-I)\varphi\|_\infty), \quad \lambda >0. $$
Choosing $\lambda = \|(L-I)\varphi\|_\infty/\|\varphi\|_\infty $ we get 
$$\| \varphi\|_{C^\beta _b(X)} \leq 2C  \|(L-I)\varphi\|_\infty^{\theta\beta} \|\varphi\|_\infty^{1-\theta\beta} \leq 2C  \| \varphi\|_{D(L)}^{\theta\beta} \|\varphi\|_\infty^{1-\theta\beta}, $$
so that $C^\beta _b(X)$ belongs to  the class $J_{\theta \beta}$ between $C_b(X)$ and $D(L)$.

 Now we fix $\beta \in (\alpha/\theta,  1/\theta)$. 
 By the Reiteration Theorem, $(C_b(X), D(L))_{\alpha, \infty}$ is continuously embedded in $(C_b(X), C^{\beta}_b(X))_{\alpha /(\theta\beta), \infty}$. In its turn, 
$(C_b(X), C^{\beta}_b(X))_{\alpha /(\theta\beta), \infty}$ is continuously embedded in  $C^{\alpha/\theta}_b(X)$ or in $Z^{\alpha/\theta}_b(X)$ by Proposition \ref{Prop:inclusione}, and the statement follows. 
\end{proof}

Under further (optimal) estimates on the moments of $\mu_t$ the spaces $(C_b(X), D(L))_{\alpha, \infty}$ may be characterized.

\begin{Proposition}
\label{pr:basic2}
Let Hypothesis \ref{Hyp:LR} hold, and assume in addition that there exist $\gamma$, $C >0$  such that  
\begin{equation}
\label{moments}
 \int_X \|x\|^\gamma \,\mu_t(dx) \leq Ct^{\gamma \theta}, \quad 0<t\leq 1. 
\end{equation}
Then for every $\alpha\in (0,  1\wedge  \theta)\cap (0, \gamma \theta]$ we have
\begin{equation}
\label{caratt1}
(C_b(X), D(L))_{\alpha, \infty} = C^{\alpha/\theta}_b(X)\cap Y_{\alpha},
 \end{equation}
and the respective norms are equivalent. 

If $\theta <1$ and moreover the measures $\mu_t$ are centered for every $t>0$, \eqref{caratt1} holds for 
every $\alpha\in (0,  1\wedge 2\theta)\cap (0, \gamma \theta]$, $\alpha\neq \theta$, still with equivalence of the respective norms. 
 %
 %
 \end{Proposition}
\begin{proof} 
Proposition \ref{pr:basic} yields the embeddings $\subset$. Concerning the embeddings $\supset$, since $\alpha/\theta  \leq \gamma$, assumption \eqref{moments} and the H\"older inequality yield
\begin{equation}
\label{stima_momento}
\int_X \|y\|^{\alpha/\theta} \,\mu_t(dy) \leq Ct^{\alpha }, \quad   0<t\leq 1.
\end{equation}

Fix $f\in  C^{\alpha/\theta}_b(X)\cap Y_{\alpha}$. 
If $\alpha <\theta $, for every $t\in (0, 1]$ and $x\in X$ using \eqref{stima_momento} we get 
$$|P_tf(x) - f(T_tx)|  \leq  \int_X |f(T_tx+y) - f(T_tx)| \,\mu_t(dy) =   \int_{X } [f]_{C^{\alpha/\theta}} \|y\|^{\alpha/\theta}   \,\mu_t(dy)
 \leq C \|f\|_{C^{\alpha/\theta}_b(X)} t^{\alpha }. $$
Since $f\in Y_\alpha$, 
\begin{equation}
\label{triang}
|P_tf(x) - f(x)|  \leq  |P_tf(x) - f(T_tx)|  + |f(T_tx) - f(x)| \leq C  \|f\|_{C^{\alpha/\theta}_b(X)}  t^{\alpha } + [f]_{Y\alpha} t^{\alpha }, \quad 0<t\leq 1, \; x\in X, 
\end{equation}
and the first statement follows from Proposition   \ref{Pr:LionsPeeetre}. 

Let now $\theta <1$ and assume that  the measures $\mu_t $ are centered for every $t>0$. Then, 
$\int_X \varphi (y)\mu_t(dy)$ $ =0$ for every $\varphi\in X^*$ and $t>0$. Take  $\alpha \in (0,1) \cap (\theta, 2\theta) \cap (0, \gamma\theta]$. 
For every  $f\in  C^{\alpha/\theta}_b(X)\cap Y_\alpha $ and  $t\in (0, 1]$,   $x\in X$, we have
$$|P_tf(x) - f(T_tx)| = \left|  \int_X f(T_tx+y) - f(T_tx) - Df(T_tx)(y)  \,\mu_t(dy) \right|, $$
so that, using the estimate  
$$|f(z+y) - f(z) - Df(z)(y) |  = \left|\int_0^1(Df(z+ \sigma y)( y)  - Df(z)(y))d\sigma \right| \leq [Df]_{C^{\alpha/\theta -1}(X, X^*)}\|y\|^{\alpha/\theta}, $$
that holds for every $ z$, $y\in X$, and \eqref{stima_momento},  we get 
$$|P_tf(x) - f(T_tx)|   \leq    [Df]_{C^{\alpha/\theta -1}(X, X^*)} \int_X \|y\|^{\alpha/\theta}\mu_t(dy) 
\leq  Ct^{\alpha }\|f\|_{C^{\alpha/\theta}_b(X)}, $$
and, as before, \eqref{triang} holds, so that $f\in (C_b(X), D(L))_{\alpha, \infty}$ by Proposition  \ref{Pr:LionsPeeetre}. 
\end{proof}

\section{Examples}
\label{Examples}

\subsection{Ornstein-Uhlenbeck operators in infinite dimension}
\label{Sect.OU}

Here $X$ is an infinite dimensional separable Banach space and the measures $\mu_t$ are Gaussian and centered. 
We refer to the book \cite{Boga} for the general theory of Gaussian measures in Banach spaces. 

Ornstein-Uhlenbeck semigroups are defined as follows. 
We fix a strongly continuous semigroup $T_t$ in $X$ generated by a linear operator $A:D(A)\subset X\mapsto X$, 
and an operator $Q\in \mathscr L( X^*, X)$ which is  non-negative (namely, $f(Qf )\geq 0$ for every $f\in X^*$) and symmetric (namely, $f(Qg) = g(Qf)$ for every $f$, $g\in X^*$). We assume that the operators $Q_t$ defined by 
$$Q_t := \int_0^t T_sQT_s^*\,ds, \quad t>0, $$
are Gaussian covariances, and for all $t>0$ we denote by $\mu_t$ the Gaussian measure with mean $0$ and covariance $Q_t$. In this case the measures $\mu_t$ satisfy \eqref{semigruppo}, 
and taking $\mu_0:=\delta_0$ the mapping $t\mapsto \mu_t$ is weakly continuous in $[0, +\infty)$. The
semigroup $P_t$ defined by  \eqref{P_t} is called {\em Ornstein-Uhlenbeck semigroup}. 

As well known, if $X$ is a Hilbert space any bounded self-adjoint, non-negative operator is the covariance of a Gaussian measure if and only if its trace is finite. 
If $X$ is just a Banach space, establishing whether a given non-negative symmetric operator in $\mathscr L( X^*, X)$ is a Gaussian covariance is not as simple;  necessary and sufficient conditions are in \cite{vNW}. 

The theory of Ornstein-Uhlenbeck operators and semigroups in infinite dimensional spaces is very rich; we refer  to the book \cite{DPZbrutto} for the basic theory  in Hilbert spaces, to the survey paper \cite{GVN} for the theory in Banach spaces, and to the more recent survey paper \cite{LP} for an up-to-date account. 

Hypothesis \ref{Hyp:LR} is satisfied provided $T_t (X)$ is contained in the Cameron-Martin space $H_t$
of $\mu_t$ for every $t>0$, and there exist $C$, $\theta >0$, $\omega \in \R$ such that 
\begin{equation}
\|T_tx\|_{H_t} \leq \frac{Ce^{\omega t}}{t^{\theta}}\|x\|, \quad t>0, \; x\in X. 
\label{holder}
\end{equation}
See  \cite[\S 5.1]{LR} for more details. 
Propositions  \ref{Pr:strong_cont} and  \ref{pr:basic} yield

\begin{Proposition}
\label{conseguenze_OU}
Assume that  \eqref{holder} holds. Then  
$$\overline{D(L)} = BUC(X)\cap Y_{0}. $$
Moreover, for every $\alpha \in (0, 1)$, 
 $$(C_b(X), D(L))_{\alpha, \infty} \subset \left\{ \begin{array}{ll}C^{\alpha/\theta}_b(X), & \alpha/\theta \notin \N, 
 \\
 \\
 Z^{\alpha/\theta}_b(X), & \alpha/\theta \in \N, 
 \end{array}\right. $$
 with continuous embedding.
\end{Proposition}

 Therefore, for any $f\in C_b(X)$ and $\alpha\in (0,1)$, $\alpha/\theta \notin \N$ we have
$$\sup_{t>0, \,x\in X}\frac{ |P_tf(x) -f(x)|}{t^\alpha} <\infty \Longrightarrow f\in C^{\alpha/\theta}_b(X), $$
while if $\alpha/\theta \in \N$, 
$$\sup_{t>0, \,x\in X}\frac{ |P_tf(x) -f(x)|}{t^\alpha} <\infty \Longrightarrow f\in Z^{\alpha/\theta}_b(X). $$

We recall that if $X$ is a Hilbert space the Cameron-Martin space of $\mu_t$ is just the range of $Q_t^{1/2}$, therefore \eqref{holder} holds if and only if  $T_t(X)\subset Q_t^{1/2}(X)$ for every $t$  and $\sup_{0<t\leq 1}t^\theta \|Q_t^{-1/2}T_t\|_{\mathcal L (X)}<+\infty$, where $Q_t^{-1/2}$ is the pseudo-inverse of $Q_t$. 

Basic examples such that  \eqref{holder} holds, with any $\theta \geq 1/2$, were given in \cite[Ex. 6.2.11]{DPZbrutto}. 
A general family of Ornstein-Uhlenbeck semigroups  such that  \eqref{holder} holds with $\theta =1/2$ was considered  in \cite{ABP}, where it was  proved that if
$t\mapsto P_tf \in C^{\alpha} ([0, +\infty); C_b(X))$ with $\alpha <1/2$,  then $f\in 
 C^{2\alpha }_b(X) $. The norm $\|f\|_\infty + [f]^{(2)}_\alpha $ was called ``semigroup norm" there.  

However, in such examples condition \eqref{moments} is not satisfied if $X$ is infinite dimensional, for any $\gamma >0$, so that Proposition \ref{pr:basic2} is not applicable. 

Instead, if $X=\R^N$ and $Q=Q^*>0$,  Hypothesis \ref{Hyp:LR} holds with $\theta = 1/2$ and \eqref{moments} holds for any $\gamma >0$. The equalities 
 $\overline{D(L)} = BUC(\R^N)\cap Y_{0}$ and $(C_b(\R^N), D(L))_{\alpha, \infty} = C^{2\alpha }_b(\R^N) \cap Y_\alpha$ for $\alpha \in (0, 1)\setminus \{1/2\}$,  $(C_b(\R^N), D(L))_{1/2, \infty} = Z^{1}_b(\R^N) \cap Y_{1/2}$ were proved in \cite{DPL}, so that applying Proposition  \ref{pr:basic2}  does not give any new information.

\subsection{Ornstein-Uhlenbeck operators with fractional diffusion in finite dimension}
\label{Sect.OU_diff_fraz}

Here we take $X= \R^N$ and we fix a symmetric positive definite matrix $Q$, a matrix $A$,  and any $s\in(0,1)$. The corresponding 
Ornstein-Uhlenbeck 
operator with fractional diffusion is given by 
\begin{equation}
\label{OUfraz}
(\mathcal L u)(x) =  \frac{1}{2}( {\text Tr}^s(QD^2u) )(x) + \langle Ax, \nabla u(x)\rangle , \quad x\in \R^N. 
\end{equation}
where  Tr$^s(QD^2)$ is the pseudo-differential operator with symbol $- \langle Q\xi, \xi\rangle^s$. 
The associated semigroup is given by (e.g., \cite{AB,PZ})
$$P_t f(x)  = \int_{\R^N} f(e^{tA}x + y) g_t(y)dy, \quad t>0, \; f\in C_b(\R^N), \; x\in \R^N, $$
where
\begin{equation}
\label{g_t}
g_t(y)  =  \frac{1}{(2\pi)^N} \int_{\R^N} e^{- \frac{1}{2}\int_0^t \|Q^{1/2}e^{\sigma A^*}\xi \|^{2s} d \sigma } e^{ -i\langle \xi, y\rangle}d\xi, \quad t>0, \;y\in \R^n. 
\end{equation}
So, $P_t$ is written in the form \eqref{P_t}, with $T_t = e^{-tA}$, $\mu_t(dy) := g_t(y)dy$, for $t>0$. For $t=0$ we set $P_0=I$ and $\mu_0=\delta_0$. 

It has been checked in  \cite[Sect. 4.2]{LR} that  $P_t$ is a generalized Mehler semigroup, and moreover   the function $g_t$ 
belongs to $W^{1,1}(\R^N)$ for every $t>0$ and 
\begin{equation}
\label{LR}
 \sup_{0<t\leq 1} t^{1/(2s)} \left\| \frac{\partial g_t}{\partial x_k}\right\|_{L^1(\R^N)} <+\infty, \quad k=1, \ldots, N, 
 \end{equation}
so that  Hypothesis \ref{Hyp:LR}  is satisfied with $\theta =1/(2s)$. Moreover, since $g_t(y) =g_t(-y)$ for every $y\in \R^N$ and $t>0$, then each $\mu_t$ is centered.  
Corollary \ref{Cor:strong_cont}  and Proposition
\ref{Pr:strong_cont} yield  

\begin{Proposition}
\label{strong_cont-s}
For every $f\in C_b(\R^N)$, we have
$$\lim_{t\to 0} \|P_tf -f\|_{\infty} =0   \Longleftrightarrow  f\in \overline{D(L)} \Longleftrightarrow f\in BUC(\R^N)\;  \text{and} \; \lim_{t\to 0}\|f(e^{tA}\cdot) -f\|_{\infty} =0. $$
\end{Proposition}

To study H\"older continuity of $t\mapsto P_tf$ through  Proposition \ref{pr:basic2} it 
remains to check that  \eqref{moments} holds, for some $\gamma >0$.

\begin{Lemma} 
\label{Le:momentiOU}
For every $\gamma < 2s$ there is $C=C(\gamma)$ such that
\begin{equation}
\label{momenti_OUfraz}
\int_{\R^N} \|y\|^\gamma \,g_t(y)dy \leq Ct^{\gamma/(2s)}, \quad 0<t\leq 1. 
\end{equation}
\end{Lemma}
\begin{proof}
By the change of variables $Q^{1/2}\xi = \eta$ in \eqref{g_t} we obtain
$$g_t(y) =    \frac{1}{(2\pi)^N(\text{Det}\,Q)^{1/2}} \int_{\R^N} e^{- \frac{1}{2}\int_0^t \| e^{\sigma B^*}\eta \|^{2s} d \sigma } e^{ -i\langle \xi, Q^{-1/2}y\rangle}d\xi,$$
with $B= Q^{1/2}AQ^{-1/2}$, and  the proof of  \eqref{momenti_OUfraz} is reduced to the case $Q=I$, just using the inverse change of variables $ Q^{-1/2}y =x$ in the left-hand side of \eqref{momenti_OUfraz}. 

So, without loss of generality we may assume $Q=I$, in which case 
$P_t$  is the transition semigroup of a stochastic differential equation,  
\begin{equation}
\label{stoch0}
dX_t = AX_t +  dL_t, \quad X_0 =x, 
\end{equation}
where $\{L_t:\;t\geq 0\}$ is a $2s$-stable standard  L\'evy process, whose laws $\nu_t$ have Fourier transforms $\widehat{\nu}_t(h)= e^{-t |h|^{2s}}$ for every $t>0$. See  \cite{PZ} for 
this explicit example, and \cite{B,S} for the general theory of L\'evy processes.  So, we have
\begin{equation}
\label{trans}
P_tf(x) = \E (f(X_t)), \quad t>0, 
\end{equation}
$X_t(x) $ being the (unique) mild solution to \eqref{stoch0}, and 
$$\int_{\R^N} |y|^\gamma g_t(y)\, dy = P_t(|\cdot|^\gamma )(0) ) = \E (|X_t(0)|^\gamma), \quad t>0. $$
Estimates for the above integrals are classical if $A=0$,  in which case $X_t = x + L_t$ and 
$\int_{\R^N} |y|^\gamma g_t(y)dy = \E (|L_t|^\gamma) \leq c t^{\gamma/(2s)}$ for $\gamma <2s$, while $\int_{\R^N} |y|^\gamma g_t(y)dy =+\infty$ for $\gamma \geq 2s$ (e.g.,  \cite[Ex. 25.10]{S}). 

In the general case,  by \eqref{stoch0} (with $x=0$) we obtain  $X_t- L_t = \int_0^t AX_\tau\,d\tau$, so that 
$|X_t| \leq  |L_t| + \|A\| \int_0^t |X_\tau|\,d\tau $,  and by the Gronwall Lemma there exists $C>0$ such that 
 $$|X_t| \leq C\sup_{0<\tau \leq t}|L_\tau|, \quad 0<t\leq 1, $$
and therefore 
\begin{equation}
\label{confronto}
\E (|X_t|^\gamma) \leq  C\E[(\sup_{0\leq \tau \leq t}|L_\tau|)^\gamma], \quad 0<t\leq 1. 
\end{equation}
To estimate the right hand side, we notice that for every $t\in (0, 1]$ we have
$$\E[\sup_{0\leq \tau \leq t}|L_\tau|]^\gamma = \E[\sup_{0\leq \tau \leq 1}|L_{t\tau}|^\gamma] = t^{\gamma/(2s )}  \E[(\sup_{0\leq \tau \leq 1} t^{-1/(2s)}|L_{t\tau}|)^\gamma ]. $$
It is known (e.g. \cite[Ch. 8]{B}) that the process $\{L^*_r:= \sup_{0\leq \tau \leq r} |L_\tau|:\; r\geq 0\}$ enjoys the scaling property of index $2s$, namely for every $k>0$ the rescaled process $\{k^{-1/(2s)}L^*_{kr}:\; r\geq 0\}$ has the same finite dimensional distributions of $\{L^*_r:\; r\geq 0\}$; in particular for every $r>0$ and $k>0$  the random variables $k^{-1/(2s)}L^*_{kr}$ and  $L^*_r$ have the same law. 
Taking $k=t$, $r=1$, we obtain 
$$ \E[(\sup_{0\leq \tau \leq 1} t^{-1/(2s)}|L_{t\tau}|)^\gamma ] =  \E[(\sup_{0\leq \tau \leq 1} |L_{\tau}|)^\gamma ]$$
and the latter is finite, due to \cite[Ex. 25.10, Thm. 25.18]{S}. Replacing in \eqref{confronto}, we get $\E (|X_t|^\gamma) \leq Ct^{\gamma/(2s )} $, and \eqref{momenti_OUfraz} follows. 
\end{proof}

Applying the results of Sections 3 and 4 we obtain the following proposition. We recall that  $Y_\alpha = \{ f\in C_b(\R^N):\; \sup_{t>0, \,x\in \R^N} t^{-\alpha}| f(e^{tA}x) - f(x)|<+\infty\}$. 

\begin{Proposition}
(i) For every $s\in (0,1)$ we have
$$\overline{D(L)} = \{ f\in BUC(\R^N):\; \lim_{t\to 0} \|f(e^{tA}\cdot )-f\|_{\infty} =0\}, $$
so that, given any $f\in C_b(\R^N)$, $P_tf$ converges uniformly to $f$ as $t\to 0$ if and only if $f\in BUC(\R^N)$ and $\lim_{t\to 0} \|f(e^{tA}\cdot) -f\|_{\infty} =0$. 

(ii) For every $s\in (0,1)$ and $\alpha \in (0, 1)\setminus \{1/(2s)\}$ we have
$$(C_b(\R^N), D(L))_{\alpha, \infty} = C^{2s\alpha}_b(\R^N) \cap Y_{\alpha} $$
so that, given any $f\in C_b(\R^N)$, we have $\sup_{t>0}t^{-\alpha}\|P_tf -f\|_{\infty} <+\infty$ if and only if $ C^{2s\alpha}_b(\R^N) \cap Y_{\alpha}$. 
%
\end{Proposition}
\begin{proof}
Statement (i) is a consequence of Corollary \ref{Cor:strong_cont} and Proposition \ref{Pr:strong_cont}. Statement (ii), with $\alpha \neq 1/(2s)$, follows from Propositions \ref{Pr:LionsPeeetre} and  \ref{pr:basic2}, since the assumptions of  Proposition \ref{pr:basic2} are satisfied with $\theta = 1/(2s)$ by the above discussion and  Lemma \ref{Le:momentiOU}. 
 \end{proof}

\begin{Remark}
{\em The results in this section hold also for $A=0$. In this case if we take $Q= 2^{s}I$ the operator $-L$ is just the realization of the fractional Laplacian $(-\Delta)^s$ in $C_b(\R^N)$. However the characterization of the interpolation spaces $D_L(\alpha, \infty)$ may be obtained for free from the general theory of powers of operators and the characterization for the Laplacian (the case of general $Q>0$  may be easily reduced to this one). }

{\em Indeed, for every Banach space $X$ and for every linear operator $T:D(T)\subset X\mapsto X$
such that $\sup_{\lambda >0}\lambda \|(\lambda I +T)^{-1}\|_{\mathcal L(X)}<+\infty$, we have the continuous embeddings  $(X, D(T))_{s, 1}\subset D(-T)^s
\subset (X, D(T))_{s, \infty}$ for every $s\in (0, 1)$. See \cite[Sect. 1.14.1]{T}, \cite[Prop. 4.7]{L1}; in both books it is assumed that $0\in\rho(T)$, but an inspection to the proof shows that this assumption is not essential. By the Reiteration Theorem we get $(X,  D(T)^s)_{\alpha, \infty} = (X, D(T))_{\alpha s, \infty}$ for every $\alpha\in (0,1)$, with equivalence of the respective norms. In our  case,  $T=-\Delta$,  we have (e.g., \cite[Thm. 3.1.12]{L})
$$(X, D(T))_{\theta, \infty} = \left\{ \begin{array}{ll}
C^{2\theta}_b(\R^N), & \theta\neq 1/2, 
\\
\\
Z^{1}_b(\R^N), & \theta = 1/2.  
\end{array}\right. $$
As a consequence, for every $\alpha \in (0, 1)$  we get
$$(X, D((-\Delta)^s)_{\alpha, \infty} = \left\{ \begin{array}{ll}
C^{2\alpha s}_b(\R^N), & \alpha \neq 1/(2s), 
\\
\\
Z^{1}_b(\R^N), &  \alpha = 1/(2s). 
\end{array}\right. $$}
\end{Remark}

 \section{Acknowledgements}
The Author is a member of GNAMPA-INdAM. 
Thanks are due to E. Priola  for suggestions about \S \ref{Sect.OU_diff_fraz}.

\end{document}